\documentclass[12pt,usenames,dvipsnames]{article}
\usepackage{hyperref}
\usepackage{cmap}
\usepackage[T1]{fontenc}\usepackage[cp1251]{inputenc}            
\usepackage[english]{babel}
\usepackage{multicol, wrapfig}
\usepackage[left=1in,right=1in,top=1in,bottom=1.5in]{geometry} 
\usepackage{amssymb,amsmath, amsthm, amscd,ifthen}
\usepackage{graphicx}
\usepackage{enumitem}
\usepackage[svgnames]{xcolor}
\usepackage{vwcol}
\usepackage{tikz}
\usetikzlibrary{patterns,decorations.text,positioning,arrows,shapes,decorations.markings}
\tikzstyle{vertex}=[circle,draw=black,fill=black,inner sep=0,minimum size=0.2cm,text=white,font=\footnotesize]

\newtheorem*{thm*}{Theorem}
\newcommand{\ff}{{\mathcal F}}
\newcommand{\aaa}{{\mathcal A}}

\newtheorem*{cla*}{Claim}
\newcommand{\bb}{{\mathcal B}}

\newcommand{\hh}{\mathcal H}
\newtheorem{thm}{Theorem}

\newtheorem{lem}[thm]{Lemma}
\newtheorem{cla}[thm]{Claim}

\newtheorem{cor}[thm]{Corollary}
\date{}

\newtheorem{prop}[thm]{Proposition}

\newtheorem{defn}{Definition}

\DeclareMathOperator{\E}{\mathrm E}

\title{Partition-free families of sets}
\author{Peter Frankl, Andrey Kupavskii\footnote{Moscow Institute of Physics and Technology, Ecole Polytechnique F\'ed\'erale de Lausanne; Email: {\tt kupavskii@yandex.ru} \ \ Research supported in part by Swiss National Science Foundation grants no. 200020-162884 and 200021-175977 and by the grant N 15-01-03530 of the Russian Foundation for Basic Research.}}

\date{}

\begin{document}
\maketitle
\begin{abstract} Let $m(n)$ denote the maximum size of a family of subsets which does not contain two disjoint sets along with their union. In 1968 Kleitman proved that $m(n) = {n\choose m+1}+\ldots +{n\choose 2m+1}$ if $n=3m+1$. Confirming the conjecture of Kleitman, we establish the same equality for the cases $n=3m$ and $n=3m+2$, and also determine all extremal families. Unlike the case $n=3m+1$, the extremal families are not unique. This is a plausible reason behind the relative difficulty of our proofs. We completely settle the case of several families as well.
\end{abstract}
\section{Introduction}
For a positive integer $n$ let $[n]:=\{1,2,\ldots,n\}$ be the standard $n$-element set and $2^{[n]}$ its power set. Subsets of $2^{[n]}$ are called \underline{families}.

In 1928 Sperner \cite{S} proved that if a family has size greater than ${n\choose \lfloor n/2\rfloor}$, then it must contain two subsets $F,G$, such that $F\subsetneq G$. This famous result served as the starting point of the presently burgeoning field of extremal set theory.

Paul Erd\H os was behind many of the early developments. In connection with an analytic problem of Littlewood and Offord he proved \cite{E46} that if $|\ff|$ is larger than the sum of the $l$ largest binomial coefficients, then $\ff$ contains a \underline{chain} $F_0\subsetneq F_1\subsetneq \ldots \subsetneq F_l.$

As much as by his results, Erd\H os also contributed to the development of extremal set theory by his many problems. Under the influence of Erd\H os, the young and promising physicist Daniel Kleitman switched to mathematics and went on to solve lots of beautiful problems. Many of these result and proofs are presented in the long chapter \cite{GK}.
For an introduction to the topic the reader is advised to consult the books \cite{A}, \cite{B}, \cite{E}, \cite{J}.

The generic extremal set theory problem is as follows. Suppose that $\ff$ does not contain a certain type of configurations. Determine or estimate the maximum of $|\ff|$. Let us give as an example the problem which is the main topic of the present paper.

The family $\ff\subset 2^{[n]}$ is called \underline{partition-free} if there are no $F_0,F_1,F_2\in \ff$ satisfying $F_1\cap F_2=\emptyset$ and $F_0 = F_1\cup F_2$. How large can $|\ff|$ be?

This problem was proposed to Kleitman by Erd\H os.
Half a century ago Kleitman \cite{Kl2} proved the following beautiful result.

\begin{thm}[Kleitman \cite{Kl2}] Suppose that $n=3m+1$ for some positive integer $m$. Let $\ff\subset 2^{[n]}$ be partition-free. Then
\begin{equation}\label{eq7} |\ff|\le \sum_{t=m+1}^{2m+1}{n\choose t}.\end{equation}
\end{thm}

\noindent \textbf{Example 1. } Let $n=3m+l, 0\le l\le 2$ and define $\mathcal K(n):= \{ K\subset [n]: m+1\le |K|\le 2m+1\}.$
It is evident that $\mathcal K(n)$ is partition-free. This shows that \eqref{eq7} is best possible.
\vskip+0.1cm

It is conjectured in \cite{Kl2} that \eqref{eq7} holds for $n=3m$ and $n=3m+2$ as well. However, for nearly half a century no progress was made on this problem.
The main purpose of the present paper is to confirm Kleitman's conjecture.

Let us mentions that Kleitman's proof is elegant and short. Unfortunately, our proof is much more technical. A reason that suggests that no easy proof exists might be that while for $n=3m+1$ $\mathcal K(n)$ from Example 1 is the \underline{unique} family attaining equality in \eqref{eq7}, it is no longer true for $n=3m+2$ and $n=3m$. \vskip+0.1cm

\noindent \textbf{Example 2. } Let $\ff\subset 2^{[n]}$ be partition-free and define $\ff^d:=\{F\subset [n+1]: F\cap [n]\in \ff\}.$
It is easy to see that $\ff^d$ is partition-free and satisfies $|\ff^d|=2|\ff|$. We call $\ff^d$ the \underline{double} of $\ff$.

Note the identity

$${3m+1\choose m+1}+{2m+1\choose m+2}+\ldots+{3m+1\choose 2m+1} = {3m+1\choose m}+\ldots+{3m+1\choose 2m},$$

$$\text{implying}\ \ \ \ \ \ \ 2\sum_{t=m+1}^{2m+1}{3m+1\choose t} = \sum_{t=m+1}^{2m+1}{3m+2\choose t}.$$
Consequently, $|\mathcal K(3m+1)^d| = |\mathcal K(3m+2)|.$\\

\noindent\textbf{Example 3. } Fix an element $x\in [n]$ and define
$$\tilde{\mathcal K}_x(n):=\Big\{F\in{[n]\choose m}:x\in F\Big\}\cup \mathcal K(n)\setminus\Big\{G\in{[n]\choose 2m+1}: x\in G \Big\}.$$
Since ${3m-1\choose m-1} = {3m-1\choose 2m}$, one has $|\tilde{\mathcal K}_x(3m)| = |\mathcal K(3m)|.$ It can be checked easily that $\tilde{\mathcal K}(n)$ is partition-free.\vskip+0.1cm


\begin{thm}\label{thm1}Suppose that $m\ge 6$ and  $n=3m+2$ or $n=3m$. If $\ff\subset 2^{[n]}$ is partition-free, then \eqref{eq7} holds. Moreover, for $n=3m+2$ the equality in \eqref{eq7} is possible (up to the permutation of the ground set) only when $\ff = \mathcal K(3m+2)$ or $\ff=\mathcal K(3m+1)^d$. For $n=3m$ the equality in \eqref{eq7} is possible only when $\ff = \mathcal K(3m)$ or $\ff=\tilde{\mathcal K}_x(3m)$ for some $x\in [n]$.
\end{thm}



Let us remark also that in view of Example 2 the inequality \eqref{eq7} for $n=3m+1$ follows from the case $n=3m+2$.

\begin{defn} Three families $\mathcal F_1,\mathcal F_2,\mathcal F_3\subset 2^{[n]}$ are called \underline{cross partition-free}, if there is no possible choice of $A\in \mathcal F_1$, $B\in \mathcal F_2$, $C\in \mathcal F_3$ such that one of those sets is equal to the disjoint union of the other two.
\end{defn}

For the case $n=3m+1$ and $n=3m+2$ one can extend \eqref{eq7} to this situation, although in the case $n=3m+2$ we get a new extremal example.

\begin{thm}\label{thm2} Suppose that $\mathcal F_1,\mathcal F_2,\mathcal F_3\subset 2^{[n]}$ are cross partition-free, $n=3m+1$ or $n=3m+2,m\ge 6$. Then
\begin{equation}\label{eq8}|\ff_1|+|\ff_2|+|\ff_3|\le 3\sum_{t=m+1}^{2m+1}{n\choose t}.\end{equation}
Moreover, for $n=3m+1$ the equality holds only when $\ff_1=\ff_2=\ff_3 = \mathcal K(3m+1)$. For $n=3m+2$ the equality up to the permutation of the indices of the families and the elements of the ground set holds only in the following three cases:  \begin{itemize} \item $\ff_1=\ff_2=\ff_3 = \mathcal K(3m+2)$, \item $\ff_1=\ff_2=\ff_3 = \mathcal K(3m+1)^d$,\item $\ff_1 = \{F\subset 2^{[n]}:m+2\le |F|\le 2m+1\},\ \ff_2=\ff_3=\{F\subset 2^{[n]}:m+1\le |F|\le 2m+2\}.$\end{itemize}
\end{thm}
Note that \eqref{eq8} implies \eqref{eq7} for $n=3m+1$, and also gives the uniqueness of the extremal family. At the same time, the $n=3m+2$ case of Theorem~\ref{thm2} implies the $n=3m+2$ case of Theorem~\ref{thm1}, along with the characterization of the extremal families.\vskip+0.1cm

For $n=3m$ one can do better.
\vskip+0.1cm
\noindent\textbf{Example 4. } Let $n=3m$ and define
\begin{align}
\mathcal A:=&\{A\subset [n]:m\le |A|\le 2m+1\}\\
\mathcal B:=\mathcal C:=&\{B\subset[n]: m+1\le |B|\le 2m\}.
\end{align}

It is easy to check that $\aaa,\bb,\mathcal C$ are cross partition-free. Using ${3m\choose m} = \frac {2m+1}m{3m\choose 2m+1}$, it follows that
$$|\aaa|+|\bb|+|\mathcal C|=3|\mathcal K(3m)|+\frac 1m{3m\choose 2m+1}.$$

\begin{thm}\label{thm3} Suppose that $\ff_1,\ff_2,\ff_3\subset 2^{[n]}$ are cross partition-free, $n=3m\ge 18$. Then
\begin{equation}\label{eq9}|\ff_1|+|\ff_2|+|\ff_3|\le 3|\mathcal K(3m)|+\frac 1m{3m\choose 2m+1}.\end{equation}
Moreover, the equality holds only for the families $\ff_i$ of the form as in Example 4.
\end{thm}

It is natural to extend the notion of partition-free to more sets. Let $r\ge 2$ be an integer. A family $\ff\subset 2^{[n]}$ said to be \underline{$r$-partition-free} if there are no pairwise disjoint members $F_1,\ldots, F_r\in\ff$ such that $F_1\cup\ldots \cup F_r\in\ff$ as well.

For $n=rm+q$, $0\le q<r$ the most natural construction of an $r$-partition-free family is:
$$\mathcal K(n,r):=\{K\subset [n]: m+1\le |K|\le rm+r-1\}.$$
In \cite{F12} it was proven that for $n=rm+r-2$ the unique optimal family is $\mathcal K(n,r)$.

However, for $r\ge 3$ the general situation is complex. It seems to be difficult to find a plausible conjecture covering all congruence classes modulo $r$. We have a few results concerning this and some related questions that will appear in \cite{FK13}.
Let us just state one of them.

An \underline{$r$-box} is a configuration consisting of $2^r-1$ subsets, namely, $r$ pairwise disjoint sets $B_1,\ldots, B_r$ along with all possible non-empty unions of them.

\begin{thm}[\cite{FK13}] Suppose that $n=rm+r-2,$ $m>r^2$ and $\ff\subset 2^{[n]}$ contains no $r$-box. Then
$$|\ff|<|\mathcal K(n,r)| \ \  \ \text{or}\ \ \ \ff=\mathcal K(n,r)\ \ \ \text{hold}.$$
\end{thm}

In the papers \cite{FK7}, \cite{FK8}, \cite{FK9} the authors advanced in related problems of Erd\H os and Kleitman on families that contain no $s$ pairwise disjoint sets.

Kleitman \cite{Kl76} considered the following related problem. What is the maximum size $u(n)$ of a family $\ff\subset 2^{[n]}$ without three distinct members satisfying $A\cup B = C$. The difference with partition-free families is that one does not require $A$ and $B$ to be disjoint.
Kleitman proves $u(n)\le {n\choose \lfloor n/2\rfloor}(1+\frac cn)$ for some absolute constant $c$.

An ``abstract'' version of this problem was solved by Katona and Tarjan \cite{KT}. Let $v(n)$ denote the maximum size of a family $\ff$ without three distinct members $A,B,C$ such that $A\subset C$ and $B \subset C$. Katona and Tarjan proved that $v(2m+1) = 2{2m\choose m}$.

This result was the starting point of a lot of research. The central problem might be stated as to determine the largest size of subsets of the boolean lattice without a certain subposet. We refer the reader to the survey \cite{GL}. One of the important recent advancements in the topic was the result of \cite{MP}, where the authors showed that for any finite poset there exists a constant $C$, such that the largest size of a family without an induced copy of this poset has size at most $C{n\choose \lfloor n/2\rfloor}$. However, the value of $C$ is unknown in most cases, including the ``diamond'' poset, and we hope that the methods developed in the present paper would be helpful to attack these problems.

Suppose that $\ff\subset 2^{[n]}$ has no three sets $A, B, C$, such that $|A\cap B|\le s$ and $A\cup B=C$. How large a family $\ff$ can be? A natural generalization of Example 1 suggests the family $\mathcal K_s(n):=\{K\subset [n]: m\le |K|<2m-s\}$ for some $m<n$, where $m$ is chosen so that the cardinality of $\mathcal K_s(n)$ is maximized.  In the discussion section we speak about how much we can advance in this problem using our methods.\\

The structure of the remaining part of the paper is as follows. In the next section we develop some of the basic tools we use. In Section 3 we prove the $n=3m+1$ case of Theorem~\ref{thm2}, which is the easiest result and which allows the reader to get familiar with some of the methods. In Section 4 we prove the $n=3m+2$ case of Theorem~\ref{thm2}, which also implies the $n=3m+2$ case of Theorem~\ref{thm1}. In Section 5 we prove the $n=3m$ case of Theorem~\ref{thm1}, which is the hardest proof in the paper. Finally, in Section 6 we prove Theorem~\ref{thm3}. In Section~7 we discuss our results and related questions.

\section{Basic tools}

For a family $\ff_i\subset 2^{[n]}$ and an integer $t$, $0\le t\le n$, we define $\ff_i^{(t)}:=\{F\in\ff:|F|=t\}$ and $f_i^{t}=|\ff_i^{(t)}|$. Let $y_i^{t}:={n\choose t}-f_i^{t}$ denote the number of $t$-sets missing from $\ff_i$. For a single family $\ff$ we use the notation $f^t, y^t$.

The following lemma is a generalization of the main lemma from Kleitman's paper \cite{Kl2}. We use the following notation: for $i\in [3]$, let $i_+=i+1,i_-=i-1,$ with $3_+=1$ and $1_-=3$ (so that we always have $\{i,i_+,i_-\}=[3]$).

\begin{lem}\label{lem1}
Suppose that $\ff_1,\ff_2,\ff_3\subset 2^{[n]}$ are cross partition-free. Let $s_1,s_2,s_3$ be nonnegative integers satisfying  $s_1+s_2+s_3 \le n$. Then the following inequality holds.
\begin{equation}\label{eq1}\sum_{i=1}^3\frac{y^{s_i}_i}{{n\choose s_i}}+\frac{y^{s_{i_+}+s_{i_-}}_i}{{n\choose s_{i_+}+s_{i_-}}}\ge 2.\end{equation}
\end{lem}

We deduce \eqref{eq1} using the following claim.
\begin{cla}\label{cla1} Let $S_1,S_2,S_3\subset [n]$ be pairwise disjoint sets satisfying $|S_i|=s_i,i=1,2,3.$ Suppose that $\ff_1,\ff_2,\ff_3\subset 2^{[n]}$ are cross partition-free. Then
\begin{equation}\label{eq11}\sum_{i=1}^3|\ff_i\cap\{S_i,S_{i_+}\cup S_{i_-}\}|\le 4.\end{equation}
\end{cla}
\begin{proof} If $S_i\in \ff_i$ holds for each $i=1,2,3,$ then $S_{i_+}\cup S_{i_-}\notin \ff_i$ for each $i$, and \eqref{eq11} holds. Now, by symmetry, we may assume that $S_3\notin \ff_3$. By the cross partition-free property, one of the relations $S_1\notin \ff_1,\ S_2\notin\ff_2,\ S_1\cup S_2\notin \ff_3$ holds, completing the proof of \eqref{eq11}.
\end{proof}

\begin{proof}[Proof of Lemma \ref{lem1}] Let us choose the pairwise disjoint sets from the claim randomly with uniform distribution. Then for $S_i\subset [n], |S_i|=s_i$, the probability of $S_i\notin \ff_i$ is $\frac {y_i^{s_i}}{{n\choose s_i}}$.  That is, the LHS of \eqref{eq1} counts the expected number of missing sets among the 6 sets $S_i, S_i\cup S_j, i,j\in[3]$. On the other hand, by Claim~\ref{cla1}, this number is always at least 2, concluding the proof.
\end{proof}

For a partition-free family $\ff\subset 2^{[n]}$ we can set $\ff_i:=\ff$ and infer:
\begin{equation}\label{eq12} \sum_{i=1}^3\frac{y^{s_i}}{{n\choose s_i}}+\frac{y^{s_{i_+}+s_{i_-}}}{{n\choose s_{i_+}+s_{i_-}}}\ge 2.\end{equation}

Kleitman \cite{Kl2} discovered this inequality and he proved \eqref{eq7} using a cleverly chosen linear combination of \eqref{eq12} for a specific choice of a set of values of $(s_1,s_2,s_3)$.
We are going to adopt this strategy for the proof of Theorem~\ref{thm2} in the case $n=3m+1$.

\begin{cor} Let $s_1,s_2,s_3$ be nonnegative integers, $s_1+s_2+s_3\le n$. Suppose that $\ff_1,\ff_2,\ff_3\subset 2^{[n]}$ are cross partition-free. Then
\begin{equation}\label{eq13} \sum_{j=1}^3\sum_{i=1}^3\frac{y^{s_i}_j}{{n\choose s_i}}+\frac{y^{s_{i_+}+s_{i_-}}_j}{{n\choose s_{i_+}+s_{i_-}}}\ge 6.\end{equation}
\end{cor}
\begin{proof}Apply \eqref{eq11} for $(s_1,s_2,s_3)$, $(s_2,s_3,s_1)$, $(s_3,s_1,s_2)$, and sum them up.\end{proof}
We are going to use the following inequality in the proofs:
\begin{equation}\label{eq005}\sum_{j=0}^{k-1}{n\choose j}< {n\choose k}\ \ \ \ \text{for any }k\le \frac n3.\end{equation}
Indeed, we have $\frac{{n\choose j}}{{n\choose j-1}} = \frac{n-j+1}{j}> 2$ for any $j\le \frac n3$, so, by the formula for the summation of a geometric progression, the inequality (\ref{eq005}) holds.

\section{The proof of Theorem~\ref{thm2} for $n=3m+1$}
Consider cross partition-free families $\ff_1,\ff_2,\ff_3\subset 2^{[n]}$.
The ideal case would be to prove an inequality of the form

\begin{equation}\label{eq14}\sum_{t=0}^n\sum_{i=1}^3\beta(t) y^{t}_i\ge 3\sum_{\substack{t\in[0,m] \cup\\ [2m+2,3m+1]}}{n\choose t}\end{equation}
 with $\beta(t)$ satisfying $\beta(t)=1$ for $0\le t\le m, 2m+2\le t\le 3m+1$ and $0\le \beta(t)<1$ for $m+1\le t\le 2m+1.$ Should we succeed with this plan, we would obtain
$$\sum_{t=0}^n\sum_{i=1}^3y^{t}_i\ge 3\sum_{\substack{t\in[0,m] \cup\\ [2m+2,3m+1]}}{n\choose t}$$
with strict inequality unless $y_j^{t}=0$ for all $t$ with $m+1\le t\le 2m+1$. That is, the only way to achieve equality is $\ff_1=\ff_2=\ff_3=\mathcal K(3m+1)$. However, we could accomplish this only partly. Namely, with $\beta(2m+1)=1$. Therefore, at the end of the proof we need to show separately that $y_i^{2m+1}=0$ for $i=1,2,3.$

We shall produce \eqref{eq14} as the sum of $m+1$ inequalities.

\begin{wraptable}{r}{0.3\linewidth}
\begin{tabular}[t]{|c|c|c|}
\hline
$s_1$&$s_2$&$s_3$\\
\hline
$m$&$m$&$m+1$\\
$m-1$&$m+1$&$m+1$\\
$m-2$&$m+1$&$m+2$\\
$m-3$&$m+1$&$m+3$\\
$\vdots$&$\vdots$&$\vdots$\\
$m-j$&$m+1$&$m+j$\\
$\vdots$&$\vdots$&$\vdots$\\
$0$&$m+1$&$2m$\\
\hline
\end{tabular}
\begin{center}Table 1\end{center}\label{tab1}
\end{wraptable}

The first one is an application of \eqref{eq13} with $s_1=s_2=m,\ s_3=m+1$. Multiplying \eqref{eq13} by $\frac 12 {n\choose m}$ we obtain $$\sum_{i=1}^3y_i^{m}+\frac{{n\choose m}}{2{n\choose m+1}}y^{m+1}_i+\frac {{n\choose m}}{2{n\choose 2m}}y_i^{2m}+y_i^{2m+1}\ge 3{n\choose m}$$
The remaining $m$ triples are listed in Table~1. We use \eqref{eq13} multiplying both sides with ${n\choose m-i}$. For $j=1$ we get
$$\sum_{i=1}^3y_i^{m-1}+y_i^{2m+2}+\frac{2{n\choose m-1}(y_i^{m+1}+y_i^{2m})}{{n\choose m+1}}\ge 6{n\choose m-1}.$$
The inequality for $j\ge 2$ is as follows
$$\sum_{i=1}^3y_i^{m-j}+y_i^{2m+j+1}+\frac{{n\choose m-1}(y_i^{m+1}+y_i^{2m})}{{n\choose m+1}}+\frac{{n\choose m-1}(y_i^{m+j}+y_i^{2m-j+1})}{{n\choose m+j}}\ge 6{n\choose m-j}.$$

Note that, since $n=3m+1$, we have
$${n\choose m+1}=\frac{2m+1}{m+1}{n\choose m}.$$
It is not difficult to see that, summing up the inequalities from Table~\ref{tab1}, the coefficients in front of each $y_i^{t}$, $m+1\le t\le 2m$, are smaller than $1$.  Indeed, for $t=m+1$ or $2m$, we have \begin{small} $$\beta(t) = \frac{\frac 12{n\choose m}+2{n\choose m-1}+1+\sum_{j=2}^m{n\choose m-j}}{{n\choose m+1}}\overset{\eqref{eq005}}{\le} \frac{\frac 12{n\choose m}+3{n\choose m-1}}{{n\choose m+1}}=\frac{\frac 12\big(1+\frac{3m}{m+1}\big){n\choose m}}{{n\choose m+1}}<\frac{\frac{2m+1}{m+1}{n\choose m}}{{n\choose m+1}}=1.$$\end{small}
For $2\le t \le m-1$ we have $\beta(m+t) = \frac{{n\choose m-t}+{n\choose t-1}}{{n\choose m+t}}$. We know that ${n\choose m+t}>{n\choose t}>2{n\choose t-1}$ and ${n\choose m+t}>{n\choose m-t+1}>2{n\choose m-t}$. Thus, ${n\choose m+t}>{n\choose t-1}+{n\choose m-t}$ and $\beta(m+t)<1$.

We also note that $\beta(t)=1$ for $t\le m$ and $t\ge 2m+1$.
Therefore, the inequality \eqref{eq8} for $n=3m+1$ is verified. Moreover, for any triple of families, for which we have equality in \eqref{eq8}, all of them must contain all the sets of sizes from $m+1$ to $2m$. We are only left to prove that in the case of equality all $(2m+1)$-sets are present in each $\ff_i$.

We are going to use the fact that for any triple of families for which equality holds in \eqref{eq8}, equality in \eqref{eq11} must hold for any choice of $S_1,S_2,S_3$.
Assume that there is a set $A, |A| = 2m+1$, which is not in $\ff_1$. Take two $m$-sets $B,C$, such that $B\cap C=\emptyset, B\cup C\subset A$. Then, since all $2m$-sets are contained in $\ff_1$, either $B\notin \ff_2$, or $C\notin \ff_3$. W.l.o.g., assume that $B\notin \ff_2$. Put $S_1:=[n]\setminus A, S_2:=B, S_3:=A\setminus B$. Then we have $S_2\cup S_3\notin \ff_1$, $S_2\notin \ff_2$. Moreover, one of the $S_1\notin\ff_1,S_3\notin\ff_3, S_1\cup S_3\notin\ff_2$ must hold. Therefore, the equality in \eqref{eq11} does not hold for this choice of $S_i$, a contradiction. This completes the proof of Theorem~\ref{thm2} in the case $n=3m+1$.

\section{The proof of Theorems \ref{thm1} and \ref{thm2} for $n=3m+2$.}


Assume that $m\ge 6$ and put $n=3m+2$ for this section. Since Theorem~\ref{thm2} implies Theorem~\ref{thm1} for $n=3m+2$, it is sufficient to prove the former. Consider cross partition-free families $\ff_1,\ff_2,\ff_3\subset 2^{[n]}$. Take three pairwise disjoint $(m-1)$-element sets $H_1^{m-1},H_2^{m-1},H_3^{m-1}$. For each such triple we define three groups of sets, indexed by $i\in[3]$, of sizes $m-1,\ldots, m+3,2m-2,\ldots, 2m+3,$  and assign them weights. Assume for simplicity that $[n]\setminus\cup_{i=1}^3 H_i^{m-1} =[5]$.

In what follows we define the $i$-th group (see Figure~\ref{figf1}). For the definition of the sets choose $j,k$ such that $\{i,j,k\}=[3]$.  The group contains (note that the upper index indicate the size of the set)
\begin{itemize}
\item one $(m-1)$-set $H^{m-1}_i$.
\item one $m$-set: $H_i^{m}:= H_i^{m-1}\cup \{i\}$,
\item two $(m+1)$-sets: for $x=4,5$ we have $H_i^{m+1}(x):=H_i^m\cup\{x\};$
\item four {\it lateral} $(m+2)$-sets: for $x=4,5$ we have $H_i^{m+2}(j,x):=H_i^m\cup\{j,x\};$
\item one {\it central} $(m+2)$-set $H_i^{m+2}:=H_i^m\cup\{4,5\}$;
\item two $(m+3)$-sets $H_i^{m+3}(j):= H_i^{m+2}\cup\{j\}$;
\item one $(2m-2)$-set: $H_i^{2m-2}:=H_{j}^{m-1}\cup H_k^{m-1}$;
\item two $(2m-1)$-sets:  $H_i^{2m-1}(j):=H_{j}^m\cup H_k^{m-1}$;
\item one $2m$-set: $H_i^{2m}:=H_{j}^m\cup H_k^m$;
\item two $(2m+1)$-sets: for $\{x,y\}=\{4,5\}$ we have $H_i^{2m+1}(y):=[n]\setminus H_i^{m+1}(x)$;
\item one $(2m+2)$-set: $H_i^{2m+2}:=[n]\setminus H_i^m$.
\item one $(2m+3)$-set: $H_i^{2m+3}:=[n]\setminus H_i^{m-1}$.
\end{itemize}
\begin{figure}[t!]
\begin{tikzpicture}[scale=1, transform shape]
\tikzstyle{every node}=[fill=white,  inner sep=2pt];
 \node(u1) at (0,0) {$H_i^{m-1}$};
 \node[fill=white](u2) at (0,1.5) {$H_i^m$} edge[thick] (u1);
\node[fill=white] at (0,0.75) {{\scriptsize $+i$}};

 \node[fill=white](u3) at (-1.5,3) {$H_i^{m+1}(4)$} edge[thick] (u2);
\node[fill=white] at (-1,2.25) {{\scriptsize $+4$}};

 \node[fill=white](u4) at (1.5,3) {$H_i^{m+1}(5)$} edge[thick] (u2);
\node[fill=white] at (+1,2.25) {{\scriptsize $+5$}};

\node[fill=white](u5) at (0,4.5) {$H_i^{m+2}$} edge[thick] (u3) edge[thick] (u4);
\node[fill=white] at (+0.8,3.75) {{\scriptsize $+4$}};
\node[fill=white] at (-0.8,3.75) {{\scriptsize $+5$}};

\node[fill=white](u6) at (-4.5,4.5) {$H_i^{m+2}(4,j)$} edge[thick] (u3);
\node[fill=white] at (-3.2,3.75) {{\scriptsize $+j$}};

\node[fill=white](u7) at (-2,4.5) {$H_i^{m+2}(4,k)$} edge[thick] (u3);
\node[fill=white] at (-1.7,3.75) {{\scriptsize $+k$}};

\node[fill=white](u8) at (+4.5,4.5) {$H_i^{m+2}(5,k)$} edge[thick] (u4);
\node[fill=white] at (+3.2,3.75) {{\scriptsize $+k$}};

\node[fill=white](u9) at (+2,4.5) {$H_i^{m+2}(5,j)$} edge[thick] (u4);
\node[fill=white] at (+1.7,3.75) {{\scriptsize $+j$}};

\node[fill=white](u10) at (+2,6) {$H_i^{m+3}(j)$} edge[thick] (u5) edge[thick] (u6) edge[thick] (u9);

\node[fill=white](u11) at (-2,6) {$H_i^{m+3}(k)$} edge[thick] (u5) edge[thick] (u7) edge[thick] (u8);
\node[fill=white] at (2,5.15) {{\scriptsize $+4$}};
\node[fill=white] at (-2,5.15) {{\scriptsize $+5$}};
\node[fill=white] at (0.75,5.05) {{\scriptsize $+j$}};
\node[fill=white] at (-0.75,5.05) {{\scriptsize $+k$}};


\node(v1) at (8,0) {$H_i^{2m+3}$};
 \node[fill=white](v2) at (8,1.5) {$H_i^{2m+2}$} edge[thick] (v1);
\node[fill=white] at (8,0.75) {{\scriptsize $-i$}};

 \node[fill=white](v3) at (6.5,3) {$H_i^{2m+1}(5)$} edge[thick] (v2);
\node[fill=white] at (7,2.25) {{\scriptsize $-4$}};

 \node[fill=white](v4) at (9.5,3) {$H_i^{2m+1}(4)$} edge[thick] (v2);
\node[fill=white] at (+9,2.25) {{\scriptsize $-5$}};

\node[fill=white](v5) at (8,4.5) {$H_i^{2m}$} edge[thick] (v3) edge[thick] (v4);
\node[fill=white] at (+8.8,3.75) {{\scriptsize $-4$}};
\node[fill=white] at (7.2,3.75) {{\scriptsize $-5$}};

\node[fill=white](v10) at (+9.5,6) {$H_i^{2m-1}(k)$} edge[thick] (v5);
\node[fill=white] at (8.75,5.26) {{\scriptsize $-j$}};

\node[fill=white](v11) at (6.5,6) {$H_i^{2m-1}(j)$} edge[thick] (v5);
\node[fill=white] at (7.25,5.25) {{\scriptsize $-k$}};
\node[fill=white](v12) at (8,7.5) {$H_i^{2m-2}$} edge[thick] (v11) edge[thick] (v10);
\node[fill=white] at (8.75,6.75) {{\scriptsize $-k$}};
\node[fill=white] at (7.25,6.75) {{\scriptsize $-j$}};

\draw[dashed,red] (u1) to[bend right] (v1);
\draw[dashed,red] (u2) to[bend right] (v2);
\draw[dashed,red] (u3) to[bend right] (v3);
\draw[dashed,red] (u4) to[bend right] (v4);
\draw[dashed,red] (u5) to[bend right] (v5);
\draw[dashed,red] (u10) to[bend right] (v10);
\draw[dashed,red] (u11) to[bend right] (v11);

\end{tikzpicture}
\caption{The family $\hh_i$. Adding/substracting the element marked on the edge from the lower set, one gets the upper set. Dashed lines connect complementary sets.}\label{figf1}
\end{figure}
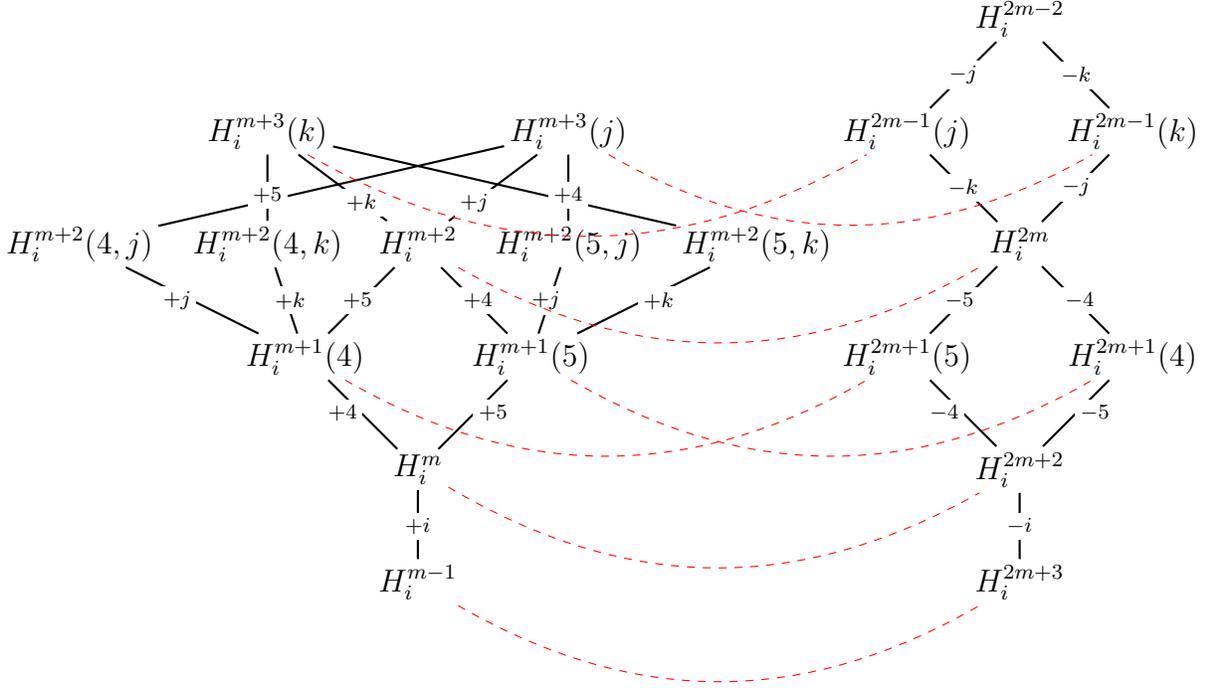

We note that $m+3<2m-2$ since $m\ge 6$.
Each set $H$ in each group gets some weight $w(H)$, which is defined by the following two conditions: the weights of two sets of the same size $j$ are the same and the total weight of $k$-sets in one group sums up to $c_k{n\choose k}$, where $c_{m+2}=c_{2m}=\frac 34, c_{m+3}=c_{2m-1}=c_{2m-2}=\frac 12$, and $c_j=1$ otherwise. E.g., the weight of each $(m+3)$-set is $\frac 1{4}{n\choose m+3}$. The only exception for the first condition are the $(m+2)$-sets, where the central set satisfies $w(H_i^{m+2})=\frac 38{n\choose m+2}$, and each lateral set satisfies $w(H^{m+2}_i(j,x))=\frac 3{32}{n\choose m+2}$. For convenience, we put $c_k=0$ for all values of $k$ not represented in the list above.

The family of all sets from the $i$-th group that have nonzero weight for a given choice of $H_i^{m-1}$ we denote $\hh_i$.


The following claim is essential for the proof.
\begin{cla}\label{clach} To prove \eqref{eq8}, it is sufficient to show that for any choice of three pairwise disjoint $(m-1)$-sets $H_i^{m-1},$ $i=1,2,3$, we have
\begin{equation}\label{eq76}\sum_{i=1}^3\sum_{F\in \ff_i\cap \mathcal H_i}w(F)\le 3\sum_{k=m+1}^{m+3}c_k{n\choose k}+3\sum_{k=2m-2}^{2m+1}c_k{n\choose k}.\end{equation}
Moreover, \eqref{eq76} implies that for any cross partition-free triple of families of maximal total size each of its families contains all sets of sizes $t\in[m+2,2m]$.
\end{cla}
\begin{proof}
For an event $A$, denote by $I[A]$ its indicator random variable. Let us take a triple of pairwise disjoint $(m-1)$-sets uniformly at random. Then for each $j\in[n]$ and $i\in [3]$ we have $$\E\bigg[\sum_{F\in \mathcal F_i\cap \mathcal H_i\cap {[n]\choose j}}w(F)\bigg] = \sum_{F\in \mathcal F_i\cap {[n]\choose j}}\E\bigg[\sum_{H\in \mathcal H_i\cap {[n]\choose j}}I[F=H]w(H)\bigg]=$$$$=\sum_{F\in \mathcal F_i\cap {[n]\choose j}}\sum_{H\in \mathcal H_i\cap {[n]\choose j}}\Pr[F=H]w(H)=\sum_{F\in \mathcal F_i\cap {[n]\choose j}}\sum_{H\in \mathcal H_i\cap {[n]\choose j}}\frac {w(H)}{{n\choose j}} = $$ $$=\sum_{F\in \mathcal F_i\cap {[n]\choose j}} c_j = c_j\Big|\ff_i\cap{[n]\choose j}\Big|=c_jf_i^j.$$

Therefore, \eqref{eq76} implies that
$$\sum_{i=1}^3\sum_{j\in[n]}c_jf_i^j = \E\bigg[\sum_{i=1}^3\sum_{F\in \mathcal F_i\cap \mathcal H_i}w(F)\bigg]\le 3\sum_{k=m+1}^{m+3}c_k{n\choose k}+3\sum_{k=2m-2}^{2m+1}c_k{n\choose k}.$$
Rewriting it in terms of $y^k_i$, we get that
\begin{equation}\label{eq2}\sum_{i=1}^3\sum_{k\in [n]}c_ky^k_i\ge 3\Big[{n\choose m-1}+{n\choose m}+{n\choose 2m+2}+{n\choose 2m+3}\Big].\end{equation}

 Let us now apply (\ref{eq13}) with $s_1 = 0,\ldots, m-2$, $s_2=m+2$, $s_3 = 2m-s_1$ (cf. Table~2), multiply each inequality by ${n\choose s_1}$ and sum the inequalities up.
 Then we get that

\begin{multline*}\sum_{i=1}^3\Bigg[\sum_{k=0}^{m-2}(y^k_i+y^{n-k}_i)+ \sum_{k=m+3}^{2m-1}\frac{{n\choose 2m-k} (y_i^k+y_i^{n-k})}{{n\choose k}}+\\ +\frac{\Big(1+2{n\choose m-2}+\sum_{k=0}^{m-3}{n\choose k}\Big)(y_i^{m+2}+y_i^{2m})}{{n\choose m+2}}\Bigg] \ge 6\sum_{k=0}^{m-2}{n\choose k}.\end{multline*}

\begin{vwcol}[widths={0.6,0.4}, rule=0pt]
(The extra $1$ and the coefficient $2$ in front of ${n\choose m-2}$ come from the terms with $s_1=0, s_1=m-2$, where $y^{m+2}$ and $y^{2m}$ appear twice.)
\vskip+0.4cm
We have $2{n\choose 2m-k}<\frac 13{n\choose k}$ for any $k=m+3,\ldots, 2m-1$, and $2{n\choose m-2}+1+\sum_{k=0}^{m-3}{n\choose k}\overset{(\ref{eq005})}{\le} 3{n\choose m-2}\le \frac 15{n\choose m+2}$. Therefore, we get that

\begin{tabular}[r]{|c|c|c|}
\hline
$s_1$&$s_2$&$s_3$\\
\hline
$m-2$&$m+2$&$m+2$\\
$m-3$&$m+2$&$m+3$\\
$\vdots$&$\vdots$&$\vdots$\\
$m-j$&$m+2$&$m+j$\\
$\vdots$&$\vdots$&$\vdots$\\
$0$&$m+2$&$2m$\\
\hline
\end{tabular}
\vskip-0.2cm
\phantom{sdlkfjsas}\qquad\qquad Table 2
\end{vwcol}
\vskip-0.2cm
\begin{equation}\label{eq3}\sum_{i=1}^3\Big[\sum_{k=0}^{m-2}(y_i^k+ y_i^{n-k})+\frac 13\sum_{k=m+3}^{2m-1}y_i^k+\frac 15(y_i^{m+2}+y_i^{2m})\Big] \ge 6\sum_{k=0}^{m-2}{n\choose k}.\end{equation}
Adding \eqref{eq2} and \eqref{eq3}, we conclude that
$$\sum_{i=1}^3\sum_{k\in[n]}\alpha_iy_i^k\ge 6\sum_{j=0}^{m}{n\choose j},$$
where $\alpha_i\le 1$, and $\alpha_i<1$ for $i\in[m+2,2m]$. This implies the inequality \eqref{eq8}, along with the second part of the claim.
\end{proof}

Let us put $\hh:=\cup_{i=1}^3\hh_i$. Our strategy to prove (\ref{eq76}) is as follows. For a set $F\in \hh_i$ we define the charge $c(F)$ to be equal to $w(F)$ if $F\in\mathcal F_i$, and to be $0$ otherwise. The {\it capacity} of $F$ is equal to $w(F)-c(F)$. Clearly, $\sum_{H\in \hh}c(H) = \sum_{F\in \mathcal F_i\cap \mathcal H_i}w(F)$. If there are no $(\le m)$- and $(\ge 2m+2)$-sets ({\it outside layers sets}) in $\ff_i\cap \hh_i$, then we are done. Otherwise, having some of those in $\ff_i\cap \hh_i$ will result in certain  sets of size $m+1\le x \le 2m+1$ ({\it middle layers sets}) not appearing in $\mathcal F_{i'}$ for $i'\ne i$.
Then we transfer (a part of) the charge of the outside layer sets to the middle layer sets with non-zero capacity. We show that the total charge transferred to each middle layer set is at most its weight. As a result of this procedure all outside layers sets will have zero total charge, and the middle layers sets will have charge not greater than their weight. This will obviously conclude the proof of the claim and \eqref{eq8}. We discuss the case of equality in \eqref{eq8} afterwards.\\


\textbf{Stage 1. Transferring charge from pairs of $\mathbf{(m-1)}$- and $\mathbf{m}$-sets}.\\ Assume that for $m_1, m_2\in \{m-1,m\}$ there are two disjoint sets $M_1, M_2\in \hh$ with non-zero charge. For definiteness say $M_1\in\ff_i, M_2\in \ff_j$, where $\{i,j,k\}=[3]$ throughout this proof. Then we transfer the charge of one of the sets to the $(m_1+m_2)$-set $M_1\cup M_2$, which is in $\hh$, but cannot be in $\ff_k$ and thus has zero charge. It is easy to see that, for any $m_1,m_2$, ${n\choose m_1}< c_{m_1+m_2}{n\choose m_1+m_2}$. We are not going to transfer any more charge to $(2m-t)$-sets, $t=0,1,2$.\\

\textbf{Stage 2. Transferring charge from $\mathbf{(m-1)}$-sets}.\\
If there remains an $(m-1)$-set, say, $H_1^{m-1}$, with non-zero charge, then in each of the four pairs $(H_{i_1}^{m+2}(1,x),H_{i_2}^{2m+1}(x)),$ $\{i_1,i_2\}=\{2,3\},\ x\in \{4,5\}$,  at least one set has zero charge. Transfer $\frac 14{n\choose m-1}$ charge to each of the sets with zero charge.\\

From now on we may assume that there are no $(m-1)$-sets and at most one $m$-set with non-zero charge in $\hh$. In the remaining part of the discharging scheme we distinguish two cases.
\vskip+0.2cm
{\scshape{ Case 1: there is an $m$-set with non-zero charge.}}
\vskip+0.2cm
\noindent In this case we assume that there is one $m$-set, say $H_1^m$, with nonzero charge. \vskip+0.1cm

\textbf{Stage 3. Transferring charge from pairs ($\mathbf{m}$-set, $\mathbf{(2m+t)}$-set), $\mathbf{t=2,3}$}.\vskip+0.1cm
 Assume that $H^{2m+t}_i$, $i,t\in \{2,3\}$, has non-zero charge. Note that $H_1^m\subset H^{2m+t}_i$. If $t=2$, then $H_i^{m+2}$ is missing from $\ff_i$ and thus has zero charge (there was no stage so far that a central $(m+2)$-set could get a charge). In that case, we transfer the charge ${n\choose m}$ of $H^{2m+2}_i$ to $H_i^{m+2}$. We have ${n\choose m}=\frac{(m+1)(m+2)}{(2m+1)(2m+2)}{n\choose m+2}<\frac38{n\choose m+2} = w(H_i^{m+2})$ for $m\ge 3$. We are not going to transfer any more charge to central $(m+2)$-sets. If $t=3$, we transfer the charge ${n\choose m-1}$ of $H^{2m+3}_i$ to $H^{m+3}_i$. Similarly, ${n\choose m-1}<\frac 12{n\choose m+3} = w(H^{m+3}_i).$\\

\textbf{Stage 4. Completing the charge transfer.}\vskip+0.1cm

The only $(2m+t)$-sets, $t=2,3$, that may still have non-zero charge, are $H^{2m+t}_1$. Let us  finish the discharging procedure in this case. To discharge $H_1^m$, consider 4 pairs of sets $(H_i^{2m+1}(x)$, $H_k^{m+1}(x))$, $\{i,k\}=\{2,3\},\ x\in\{4,5\}$,  in each of which at least one set must be missing from the corresponding $\ff_{i'}, i'\in\{2,3\}$. We transfer $\frac 14{n\choose m}$ charge to (one of) the missing set  in each pair. Note that the total charge of $H_i^{2m+1}(x)$ after this stage is at most $\frac 14({n\choose m}+{n\choose m-1})<\frac 12{n\choose m+1} = w(H_i^{2m+1}(x))$. The $(2m+1)$-sets are not going to get any more charge.

Next, we transfer the charge from $H^{2m+2}_1$. Choosing $x,y$ such that $\{x,y\}=\{4,5\}$, we see that in each of the two pairs $(H^{m+1}_2(x),H^{m+1}_3(y))$ there is at least one set missing from the corresponding $\ff_{i'}$. We transfer to each of them $\frac 12 {n\choose m}$ charge.

Similarly, if $H^{2m+3}_1$ has non-zero charge, then in one of the $4$ pairs $(H_i^{m+1}(x),\ H_k^{m+2}(i,y))$, $\{x,y\}=\{4,5\}$ and $ \{i,k\}=\{2,3\}$,  one set is missing from the corresponding $\ff_{i'}$. We transfer $\frac 14{n\choose m-1}$ charge to each of the missing sets. At this point there is no set in $\ff_i\cap \hh$ of size $(\le m)$ or $(\ge 2m+2)$ that has non-zero charge for $i=1,2,3$. To conclude the proof for Case 1, we have to show that the $(m+1)$- and $(m+2)$-sets did not get overcharged. The charge of any missing $(m+1)$-set was zero until Stage 4, and is at most $\frac 14{n\choose m}+\frac 12{n\choose m}+\frac 14{n\choose m-1}<{n\choose m}<\frac 12{n\choose m+1}$, which is the weight of any $(m+1)$-set. The charge of any lateral $(m+2)$-set is at most $\frac 12{n\choose m-1}$ (it could have increased at Stage 2, and at the last part of Stage 4). We have $$\frac 12 {n\choose m-1} = \frac{m(m+1)(m+2)}{2(2m+1)(2m+2)(2m+3)}{n\choose m+2}= \frac{m(m+2)}{4(2m+1)(2m+3)}{n\choose m+2}= $$$$ \frac{(m+1)^2-1}{16((m+1)^2-\frac 14)}{n\choose m+2}< \frac 1{16}{n\choose m+2}.$$ This is less than $\frac 3{32} {n\choose m+2}$, which is the charge of any lateral $(m+2)$-set. The other sets did not get any extra charge at Stage 4, and had less charge than weight at earlier stages. The proof of \eqref{eq8} is complete in Case 1.\\

{\scshape{ Case 2: there is no $m$-set with non-zero charge.}}
\vskip+0.2cm
\noindent In this case we assume that there was no $m$-set with non-zero charge left after Stage 2. Remark that any missing $(m+1)$-set has zero charge at this stage. \vskip+0.1cm
\textbf{Stage 3. Transferring charge from $\mathbf{(2m+3)}$-sets.}

Assume that there is a $(2m+3)$-set  $H^{2m+3}_i$ with non-zero charge. Then in one of the $4$ pairs $(H_j^{m+1}(x),\ H_k^{m+2}(i,y))$, where $\{x,y\}=\{4,5\}$ and $ \{i,j,k\}=[3]$, one set is missing from the corresponding $\ff_{i'}$. We transfer $\frac 12{n\choose m-1}$ charge to each of the missing $(m+2)$-sets, and the rest of the charge distribute evenly between the missing $(m+1)$-sets. Remark that a lateral $(m+2)$-set may get charge from one $(2m+3)$-set only. So if a lateral $(m+2)$-set was missing, then it has at most $\frac 34{n\choose m-1}<  \frac 3{32}{n\choose m+2}$  charge after this stage. Therefore, its charge is strictly smaller than its weight. We are not going to transfer any more charge to lateral $(m+2)$-sets.\\

\textbf{Stage 4. Transferring charge from $\mathbf{(2m+2)}$-sets.}

Next, we transfer the charge from $H^{2m+2}_i$. Choosing $x,y,j,k$ such that $\{x,y\}=\{4,5\}$, $\{i,j,k\}=[3]$, we get that in each pair $(H^{m+1}_j(x),H^{m+1}_k(y))$ there is at least one set missing from the corresponding $\ff_{i'}$. We transfer the charge of the $(2m+2)$-sets to the missing $(m+1)$-sets. Assume that there are $k_{m+1}$ $(m+1)$-sets that are missing (i.e., belong to $\cup_{i=1}^3(\hh_i-\ff_i)$) and that $k_{2m+2}$ $(2m+2)$-sets in $\hh$ have non-zero charge. Moreover, assume that the total charge of $k_{2m+3}{n\choose 2m+3}$ from $(2m+3)$-sets was transferred to $(m+1)$-sets (note that $k_{2m+3}$ may be half-integer since a part of the charge of some $(2m+3)$-sets could have been transferred to the missing $(m+2)$-sets). Then we need to make sure that
\begin{equation}\label{disc} k_{m+1}\ge k_{2m+2}+\frac 12 k_{2m+3}\end{equation}
to complete the proof of \eqref{eq76}. Indeed, the capacity of each $(m+1)$-set is $\frac 12{n\choose m+1}$, while the charge of a $(2m+2+j)$-set is ${n\choose m-j}$. Therefore, if \eqref{disc} holds, we get that the capacity of $(m+1)$-sets is bigger than the charge transferred to them:
$$\frac {k_{m+1}}2{n\choose m+1}= k_{m+1}{n\choose m}\ge k_{2m+2}{n\choose m}+2(k_{m+1}-k_{2m+2}){n\choose m-1}\ge$$$$ k_{2m+2}{n\choose m}+k_{2m+3}{n\choose m-1}.$$
The first inequality above holds since $k_{m+1}\ge k_{2m+2}$ and $2{n\choose m-1}<{n\choose m}$. Note that it may be replaced by a strict inequality, if \eqref{disc} holds and $k_{2m+3}>0$. The second inequality holds due to \eqref{disc}.

We note the following useful fact: if we have $j$ sets, $j\in\{1,2\}$, of size $m+1$ in $\cup_{i=1}^3(\ff_i\cap \hh_i)$, which are contained in $H^{2m+3}_k$, then $\frac j2$ of the charge of $H^{2m+3}_k$ is transferred to the $(m+2)$-sets and thus is not transferred to the $(m+1)$-sets. In particular, this implies that
\begin{equation}\label{dependence} k_{m+1}\le 5\ \Rightarrow\ k_{2m+3}\le 2\ \ \ \ \ \ \text{and} \ \ \ \ \ \ k_{m+1}\le 2\ \Rightarrow\ k_{2m+3}=0.\end{equation}
Below we consider two subcases.\\

\textbf{Case A: three $\mathbf{(2m+2)}$-sets with non-zero charge.} Having all three $(2m+2)$-sets in $\ff\cap \hh$ implies that in every pair of disjoint $(m+1)$-sets one is missing from the corresponding $\ff_{i'}$, which means that the $(m+1)$-sets $H^{m+1}_i(x)$ from $\cup_{i=1}^3(\ff_i\cap \hh_i)$ have either all the same $i$, or the same $x$.

If they all have the same $x$, then the number of $(m+1)$-sets can be at most $3$.
If there are exactly $3$ sets, then, for each $i$, $H^{2m+3}_i$ contains at least two $(m+2)$-sets from $\cup_{i=1}^3(\hh_i\setminus\ff_i)$, and so all the charge of $(2m+3)$-sets is transferred to $(m+2)$-sets. Thus, in this case we have $k_{2m+3}=0$ and $3=k_{m+1}= k_{2m+2}+\frac 12 k_{2m+3}$.

If there are at most two sets with the same $x$, or the $(m+1)$-sets have the same $i$, then $k_{m+1}\ge 4$ and, using \eqref{dependence}, we get \eqref{disc} again.\\

\textbf{Case B: one or two $\mathbf{(2m+2)}$-sets with non-zero charge.}
Having at least one $(2m+2)$-set in $\ff\cap \hh$ implies that $k_{m+1}\ge 2$. If $k_{m+1}=2$, then $k_{2m+3}=0$ by \eqref{dependence}, and \eqref{disc} holds. If $3\le k_{m+1}\le 5$, then $k_{2m+3}\le 2$ and \eqref{disc} also holds. If $k_{m+1}=6$, then \eqref{disc} holds again.

We have verified that the inequality \eqref{disc} holds always. This implies that we have fulfilled all the condition imposed on the charging and discharging schemes. The proof of the inequality \eqref{eq8} for $n=3m+2$ is complete.\\

\textbf{Extremal families}\vskip+0.1cm
We are only left to analyze the families attaining equality in \eqref{eq8}. (We call such triples of families \underline{extremal}.) By Claim~\ref{clach}, $\ff_i\supset\cup_{k=m+2}^{2m}{[n]\choose k}$ for each $i\in [3]$. During the charging-discharging process none of the missing sets of size $2m+1$ got fully charged. Therefore, $\ff_i\supset {[n]\choose 2m+1}$.

Let us further analyze the scenarios in which all sets in $\hh$ got fully charged. To achieve this, we have to fall into Case 2 and get an equality in \eqref{disc} with $k_{2m+3}=0$. Moreover, we cannot have any $(2m+3)$-sets in the family either, since this causes some $(m+2)$-sets to be missing from one of the families. We also infer that none of the sets of sizes $k\ge 2m+3$ and $k\le m$ are in the families (otherwise, one of the sets of size in $[m+2,2m+1]$ is missing from one of the families).

Therefore, $$\bigcup_{k=m+2}^{2m+1}{[n]\choose k}\subset \ff_i\subset \bigcup_{k=m+1}^{2m+2}{[n]\choose k}$$ for each $i\in[3]$, and we have the following three possibilities:
\begin{itemize}
\item[i] We fall into Case A and $k_{m+1}=k_{2m+2}=3$. It means that all three $(2m+2)$-sets are present in $\cup_{i=1}^3(\ff_i\cap\hh_i)$ and that for some $x\in [n]\setminus(\cup_{i=1}^3 H_i^m)$ none of the $H_i^{m+1}(x)$ are in $\cup_{i=1}^3(\ff_i\cap\hh_i)$.
\item[ii] We fall into Case B and $k_{m+1}=k_{2m+2}=2.$ Then for some $i\in[3]$, say, $i=1$, $\mathcal F_1\cap \hh_1$ does not contain $(m+1)$- and $(2m+2)$-sets, while both $\mathcal F_{2}\cap \hh_{2}$ and  $\mathcal F_{3}\cap \hh_{3}$ contain all possible $(m+1)$- and $(2m+2)$-sets.
\item[iii] We fall into Case B and have $k_{m+1}=k_{2m+2}=0$, which means that none of the three possible $(2m+2)$-sets are present in $\cup_{i=1}^3(\ff_i\cap\hh_i)$, while all $(m+1)$-sets are.
\end{itemize}
To conclude the proof, we need to analyze these possibilities and to show that for three cross partition-free families of maximum total size the same option holds for {\it all choices of triples simultaneously}. Then Option i leads to $\ff_1=\ff_2=\ff_3=\mathcal K(3m+1)^d$, Option iii leads to $\ff_1=\ff_2=\ff_3=\mathcal K(3m+2)$, and Option ii yields $\ff_1 = \{F\subset 2^{[n]}:m+2\le |F|\le 2m+1\},\ \ff_2=\ff_3=\{F\subset 2^{[n]}:m+1\le |F|\le 2m+2\}.$\\

Assume that for a given triple $F_1,F_2,F_3$ of $m$-sets with $\{x,y\}:=[n]\setminus\cup_{i=1}^3F_i$ Option i holds, and, say, $F_i\cup\{x\}$ belong to $\ff_i\cap\hh_i$, while $F_i\cup\{y\}$ does not. We aim to show that in this situation all $(m+1)$-sets containing $y$ are missing from each $\ff_i$ (and thus $y$ plays the role of the last element in the definition of the doubling of a family).

First of all, let us show that for any $F\in{[n]\setminus\{x,y\}\choose m}$ and $i\in[3]$, the set $F\cup\{x\}$ belongs to $\ff_i$, and $F\cup\{y\}$ does not. Indeed, consider a set $F'\in{[n]\setminus(\{x,y\}\cup F_i\cup F)\choose m}$. Then $F'$ and $F$ together with a third $m$-set form a triple that is of type i. (Indeed, in other options it is impossible to have $F_i\cup\{x\} \in \ff_i, F_i\cup\{y\}\notin \ff_i$.) Therefore, $F'\cup\{x\}\in \ff_{i^+}$, and $F'\cup\{y\}\notin\ff_{i^+}$. Applying the same argument again to a triple formed by $F'$ and $F_i$, we get that $F\cup\{x\}\in \ff_{i}$, and $F\cup\{y\}\notin\ff_{i}.$ This also implies that for each $i\in [3]$ any $(2m+2)$-set that contains both $x$ and $y$ belongs to $\ff_i$.

Next, we aim to show that the $(m+1)$-sets that contain both $x$ and $y$ are missing from each $\ff_i$. Assume the contrary, that it, that there is a set $H_1\in {[n]\choose m}$, $x\in H_1, y\notin H_1$, such that $H_1\cup \{y\}\in \ff_1$. But we also know that for any $z\in [n]\setminus (H_1\cup \{y\}$ the set $H_1\cup \{z\}$ is in $\ff_i$. Consider a partition of $[n]\setminus (H_1\cup\{y,z\})$ into two $m$-sets $H_2,H_3$. Then $H_1,H_2,H_3$ form a triple, in which both $H_1\cup\{y\}$ and $H_1\cup\{z\}$ belong to $\ff_1$. But on the other hand, we know that $[n]\setminus H_2\in \ff_2$ and $[n]\setminus H_3\in \ff_3$, since both $(2m+2)$-sets contain $x$ and $y$. But then all $(m+1)$-sets $H_j\cup\{z\}, H_j\cup\{y\}$ must be missing from $\ff_j$, $j=2,3$, so this triple is not one of the types i-iii, a contradiction.\\

The last step is to prove that we cannot have both Option ii and Option iii for different triples of $m$-sets for the same triple of families.  Assume that for some extremal families we have one choice of a triple of $m$-sets, for which Option ii holds. We claim that in this case we have $\ff_1\cap {[n]\choose m+1}=\emptyset, \ff_2,\ff_3\supset {[n]\choose 2m+2}$. Indeed, assume that $F'\in \ff_1\cap {[n]\choose m+1}$, $H'\notin\ff_1\cap {[n]\choose m+1}$. Then, applying a standard ``continuity'' argument, we get that there exist two sets $F,H\in {[n]\choose m+1}$ and elements $x,y\in[n]$, such that $F\Delta H = \{x,y\}$, and $F\in\ff_1, H\notin\ff_1$. Put $F_1:=F\cap H$, and choose a partition $F_2,F_3$ of $[n]\setminus (F\cup H)$ into two $m$-element sets. Then the triple $F_1,F_2,F_3$ is neither of type ii, nor of type iii, a contradiction. This concludes the proof of Theorem~\ref{thm2}.

\section{The proof of Theorem \ref{thm1} for $n=3m$.}


Assume that $m\ge 6$ and fix $n=3m$ for this section. The proof of this part of the theorem is similar in spirit to the proof of the previous part, but the family $\hh$ is substantially different. The family $\hh$ in this proof is invariant under the action of the cyclic group of order $n$. Let us mention that the usefulness of the cycle for the extremal set theory problems was first discovered by Katona \cite{Ka1}. Consider a family $\ff\subset 2^{[n]}$ satisfying the requirements of the theorem.

Fix a cyclic permutation $\sigma$, and redefine for simplicity $i:=\sigma(i)$. Put $$I:= [m-2,m+1]\cup [2m-1,2m+3].$$ We consider a weighted family $\mathcal H$ of sets associated with $\sigma$, containing the following sets (note that all additions and substractions are made modulo $n$; see Figures~\ref{fig1} and~\ref{fig2} for illustration):
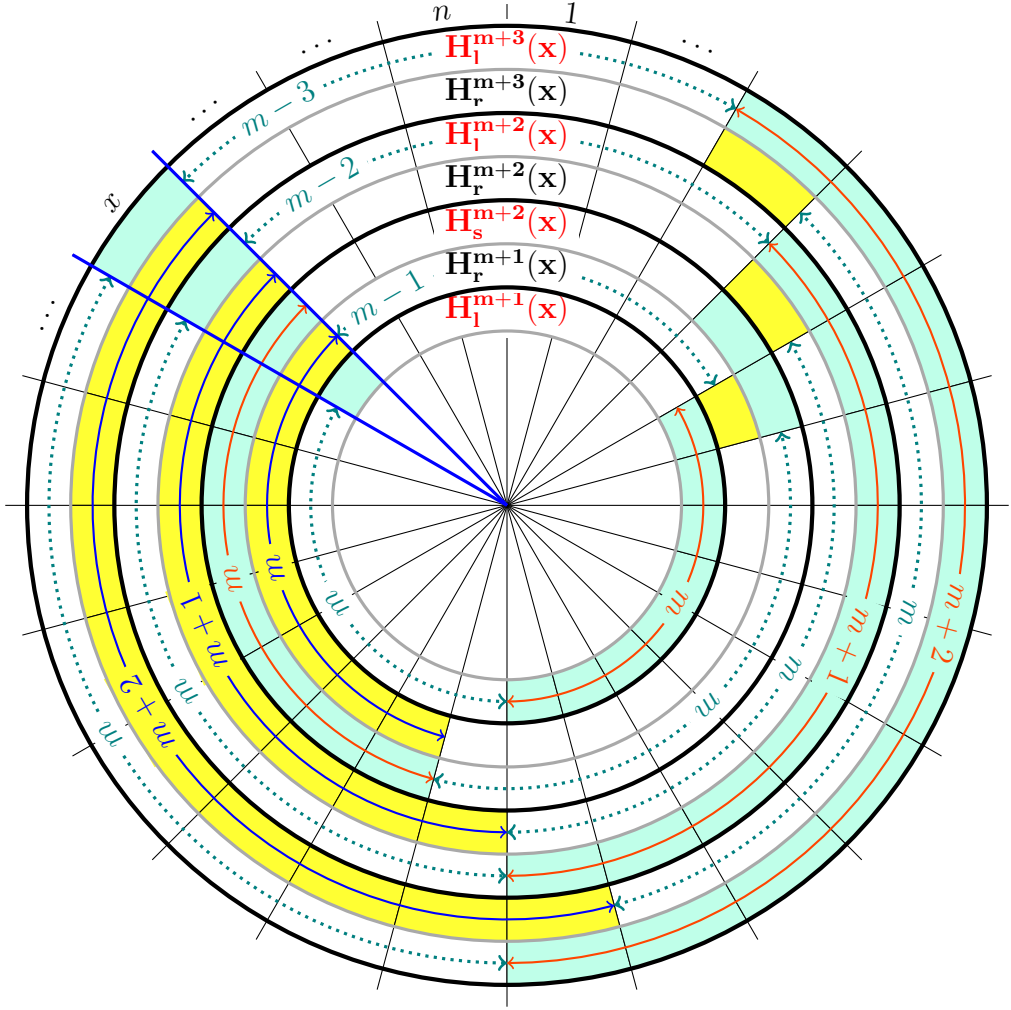
\begin{figure}\centering
\begin{tikzpicture}[scale=0.58]

\foreach \x in {1,...,24}
\draw (0,0) -- (15*\x:11.5);



\foreach \x in {1,0,2,3,9,18,19,20,21,22,23}
\filldraw[fill=Aquamarine!50, draw=black, opacity=1] (15*\x:10) -- (15*\x:11) arc (15*\x:15*(\x+1):11) -- (15*\x+15:10) arc (15*\x+15:15*\x:10);
\draw [->, OrangeRed,thick,domain=-15:60] plot ({10.5*cos(\x)}, {10.5*sin(\x)});
\draw [<-, OrangeRed,thick,domain=-90:-15] plot ({10.5*cos(\x)}, {10.5*sin(\x)})
node[anchor=center,fill=Aquamarine!50,  rotate=-105]{$m+2$};
\draw [dotted,->, Teal,very thick,domain=-150:-90] plot ({10.5*cos(\x)}, {10.5*sin(\x)});
\draw [dotted, <-, Teal,very thick,domain=-210:-150] plot ({10.5*cos(\x)}, {10.5*sin(\x)})
node[anchor=center,fill=white,  rotate=-230]{$m$};
\draw [dotted,->, Teal,very thick,domain= 120:135] plot ({10.5*cos(\x)}, {10.5*sin(\x)});
\draw [dotted, <-, Teal,very thick,domain=60:120] plot ({10.5*cos(\x)}, {10.5*sin(\x)})
node[anchor=center,fill=white,  rotate=30]{$m-3$};
\node[fill=white] at (0,10.5) {$\textcolor{red}{\mathbf{H^{m+3}_l(x)}}$};



\foreach \x in {9,10,11,12,13,14,15,16,17,18,3}
\filldraw[fill=Yellow!80, draw=black, opacity=1] (15*\x:9) -- (15*\x:10) arc (15*\x:15*(\x+1):10) -- (15*\x+15:9) arc (15*\x+15:15*\x:9);
\draw [dotted,->, Teal,very thick,domain=-15:45] plot ({9.5*cos(\x)}, {9.5*sin(\x)});
\draw [dotted,<-, Teal,very thick,domain=-75:-15] plot ({9.5*cos(\x)}, {9.5*sin(\x)})
node[anchor=center,fill=white,  rotate=-105]{$m$};
\draw [->, Blue,thick,domain=-150:-75] plot ({9.5*cos(\x)}, {9.5*sin(\x)});
\draw [<-, Blue,thick,domain=-225:-150] plot ({9.5*cos(\x)}, {9.5*sin(\x)})
node[anchor=center,fill=Yellow!80,  rotate=-240]{$m+2$};
\node[fill=white] at (0,9.5) {$\textcolor{black}{\mathbf{H^{m+3}_r(x)}}$};



\foreach \x in {1,0,2,9,18,19,20,21,22,23}
\filldraw[fill=Aquamarine!50, draw=black, opacity=1] (15*\x:8) -- (15*\x:9) arc (15*\x:15*(\x+1):9) -- (15*\x+15:8) arc (15*\x+15:15*\x:8);
\draw [->, OrangeRed,thick,domain=-22.5:45] plot ({8.5*cos(\x)}, {8.5*sin(\x)});
\draw [<-, OrangeRed,thick,domain=-90:-22.5] plot ({8.5*cos(\x)}, {8.5*sin(\x)})
node[anchor=center,fill=Aquamarine!50,  rotate=-112.5]{$m+1$};
\draw [dotted,->, Teal,very thick,domain=-150:-90] plot ({8.5*cos(\x)}, {8.5*sin(\x)});
\draw [dotted, <-, Teal,very thick,domain=-210:-150] plot ({8.5*cos(\x)}, {8.5*sin(\x)})
node[anchor=center,fill=white,  rotate=-230]{$m$};
\draw [dotted,->, Teal,very thick,domain= 120:135] plot ({8.5*cos(\x)}, {8.5*sin(\x)});
\draw [dotted, <-, Teal,very thick,domain=45:120] plot ({8.5*cos(\x)}, {8.5*sin(\x)})
node[anchor=center,fill=white,  rotate=30]{$m-2$};
\node[fill=white] at (0,8.5) {$\textcolor{red}{\mathbf{H^{m+2}_l(x)}}$};


\foreach \x in {9,10,11,12,13,14,15,16,17,2}
\filldraw[fill=Yellow!80, draw=black, opacity=1] (15*\x:7) -- (15*\x:8) arc (15*\x:15*(\x+1):8) -- (15*\x+15:7) arc (15*\x+15:15*\x:7);
\draw [dotted,->, Teal,very thick,domain=-30:30] plot ({7.5*cos(\x)}, {7.5*sin(\x)});
\draw [dotted,<-, Teal,very thick,domain=-90:-30] plot ({7.5*cos(\x)}, {7.5*sin(\x)})
node[anchor=center,fill=white,  rotate=-120]{$m$};
\draw [->, Blue,thick,domain=-157.5:-90] plot ({7.5*cos(\x)}, {7.5*sin(\x)});
\draw [<-, Blue,thick,domain=-225:-157.5] plot ({7.5*cos(\x)}, {7.5*sin(\x)})
node[anchor=center,fill=Yellow!80,  rotate=-247.5]{$m+1$};
\node[fill=white] at (0,7.5) {$\textcolor{black}{\mathbf{H^{m+2}_r(x)}}$};



\foreach \x in {1,0,9,18,19,20,21,22,23}
\filldraw[fill=Aquamarine!50, draw=black, opacity=1] (15*\x:4) -- (15*\x:5) arc (15*\x:15*(\x+1):5) -- (15*\x+15:4) arc (15*\x+15:15*\x:4);
\draw [->, OrangeRed,thick,domain=-30:30] plot ({4.5*cos(\x)}, {4.5*sin(\x)});
\draw [<-, OrangeRed,thick,domain=-90:-30] plot ({4.5*cos(\x)}, {4.5*sin(\x)})
node[anchor=center,fill=Aquamarine!50,  rotate=-120]{$m$};
\draw [dotted,->, Teal,very thick,domain=-150:-90] plot ({4.5*cos(\x)}, {4.5*sin(\x)});
\draw [dotted, <-, Teal,very thick,domain=-210:-150] plot ({4.5*cos(\x)}, {4.5*sin(\x)})
node[anchor=center,fill=white,  rotate=-230]{$m$};
\node[fill=white] at (0,4.5) {$\textcolor{red}{\mathbf{H^{m+1}_l(x)}}$};



\foreach \x in {9,10,11,12,13,14,15,16,1}
\filldraw[fill=Yellow!80, draw=black, opacity=1] (15*\x:5) -- (15*\x:6) arc (15*\x:15*(\x+1):6) -- (15*\x+15:5) arc (15*\x+15:15*\x:5);
\draw [->, Blue,thick,domain=-165:-105] plot ({5.5*cos(\x)}, {5.5*sin(\x)});
\draw [<-, Blue,thick,domain=-225:-165] plot ({5.5*cos(\x)}, {5.5*sin(\x)})
node[anchor=center,fill=Yellow!80,  rotate=-255]{$m$};
\draw [dotted,->, Teal,very thick,domain= 120:135] plot ({5.5*cos(\x)}, {5.5*sin(\x)});
\draw [dotted, <-, Teal,very thick,domain=30:120] plot ({5.5*cos(\x)}, {5.5*sin(\x)})
node[anchor=center,fill=white,  rotate=30]{$m-1$};
\node[fill=white] at (0,5.5) {$\textcolor{black}{\mathbf{H^{m+1}_r(x)}}$};


\foreach \x in {9,10,11,12,13,14,15,16,1,2}
\filldraw[fill=Aquamarine!50, draw=black, opacity=1] (15*\x:6) -- (15*\x:7) arc (15*\x:15*\x+15:7) -- (15*\x+15:6) arc (15*\x+15:15*\x:6);
\draw [dotted,->, Teal,very thick,domain=-45:15] plot ({6.5*cos(\x)}, {6.5*sin(\x)});
\draw [dotted,<-, Teal,very thick,domain=-105:-45] plot ({6.5*cos(\x)}, {6.5*sin(\x)})
node[anchor=center,fill=white,  rotate=-135]{$m$};
\draw [->, OrangeRed,thick,domain=-165:-105] plot ({6.5*cos(\x)}, {6.5*sin(\x)});
\draw [<-, OrangeRed,thick,domain=-225:-165] plot ({6.5*cos(\x)}, {6.5*sin(\x)})
node[anchor=center,fill=Aquamarine!50,  rotate=-255]{$m$};
\node[fill=white] at (0,6.5) {$\textcolor{red}{\mathbf{H^{m+2}_s(x)}}$};


\foreach \x in {5,7,9,11}
\draw[ultra thick] (0,0) circle (\x);
\foreach \x in {4,6,8,10}
\draw[very thick,DarkGrey] (0,0) circle (\x);


\node[rotate=-7.5] at ({11.4*sin(7.5)},{cos(7.5)*11.4}) { $1$};
\node[rotate=7.5] at ({11.4*sin(-7.5)},{cos(7.5)*11.4}) { $n$};
\node[rotate=22.5] at ({11.4*sin(-22.5)},{cos(22.5)*11.4}) { $\ldots$};
\node[rotate=-22.5] at ({11.4*sin(22.5)},{cos(22.5)*11.4}) { $\ldots$};
\node[rotate=52.5] at ({11.4*sin(-52.5)},{cos(52.5)*11.4}) { $x$};
\node[rotate=37.5] at ({11.4*sin(-37.5)},{cos(37.5)*11.4}) { $\ldots$};
\node[rotate=67.5] at ({11.4*sin(-67.5)},{cos(67.5)*11.4}) { $\ldots$};

\foreach \x in {9,10}
\draw[very thick, blue] (0,0) -- (15*\x:11.5);
\end{tikzpicture}
\caption{Non-interval sets from the family $\hh$ of sizes from $m+1$ to $m+3$. See the digression on how to read figures for the interpretation.}\label{fig1}
\end{figure}
\begin{itemize}[leftmargin= 2cm]
\item[{\scriptsize\it $\mathbf{I}$:}] $n$ {\it interval} sets of size $j$, $j\in I$: for each $x\in[n]$ put $H^{j}(x):= [x-j+1, x]$, with the weight  $w(H^{j}(x))$ satisfying
    $$\mathbf{w(H^{j}(x)):={n\choose j}}\ \ \ \text{ for } \ \ \ j\in I-\{m+1,2m-1\}.$$
    We also have
    $$\mathbf{w(H^{m+1}(x))} := 2{n\choose m-1} \mathbf{= \frac {m+1}{2m+1}{n\choose m+1}},\ \ \ \mathbf{w(H^{2m-1}(x)) := \frac 27{n\choose m+1}}.$$
\item[{\scriptsize $\mathbf{m+1}$:}] $2n$ sets of size $m+1$: for each $x\in[n]$ put $H_l^{m+1}(x) = [x-2m,x-m-1]\cup \{x\}$ and $H_r^{m+1}(x) = [x-m+1, x]\cup \{x-2m\}$. We have $$\mathbf{w(H_l^{m+1}(x)) := w(H_r^{m+1}(x)) := \frac m{4m+2}{n\choose m+1}} = \frac{m}{m+1}{n\choose m-1}.$$ We call $H_l^{m+1}(x), H_r^{m+1}(x),$ and $H^{m+1}(x)$ \underline {left}, \underline{right}, and \underline{central} $(m+1)$-sets, respectively.
\item[{\scriptsize $\mathbf{m+2,m+3}$:}] $2n$ sets of size $m+j$, $j=2,3$: for each $x\in [n]$ put $H_l^{m+j}(x) = [x-2m-j+1,x-m-1]\cup \{x\}$ and $H_r^{m+j}(x) = [x-m+2-j, x]\cup \{x-2m-j+1\}$. We put
    $$\mathbf{w(H_l^{m+j}(x)) := w(H_r^{m+j}(x)) := \frac 16{n\choose m+j}} \ \ \ \text{ for } \ \ \ j=2,3.$$
\item[{\scriptsize $\mathbf{m+2}$:}] $n$ sets of size $m+2$: for each $x\in [n]$ put $H_s^{m+2}(x) = [x-m+1,x]\cup \{x-2m-1,x-2m\}$. We put
    $$\mathbf{w(H_s^{m+2}(x)):= \frac 16{n\choose m+2}}.$$

\item[{\scriptsize $\mathbf{2m-2}$:}] $2n$ sets of size $2m-2$:  $H_a^{2m-2}(x):=H^m(x)\cup H^{m-2}(x-m-1)$ and $H_b^{2m-2}(x):=H^{m-1}(x)\cup H^{m-1}(x-m)$; we put
    $$\mathbf{w\big(H_a^{2m-2}(x)\big) := w\big(H_b^{2m-2}(x)\big) := \frac 14{n\choose 2m-2}.}$$
\item[{\scriptsize $\mathbf{2m-1}$:}] $2n$ sets of size $2m-1$:  $H_r^{2m-1}(x):=H^m(x)\cup H^{m-1}(x-m-1)$ and $H_l^{2m-1}(x):=H^{m-1}(x)\cup H^{m}(x-m)$; we put
    $$\mathbf{w(H_l^{2m-1}(x)) := w(H_r^{2m-1}(x)):= \frac 27{n\choose 2m-1}.}$$
\end{itemize}

\vskip+0.2cm

{\usefont{T1}{cmr}{b}{sc} Digression. How to read figures.} {\usefont{T1}{cmr}{m}{sl} The definition of $\hh$ and the proof is quite technical and is based on the relationships between different sets. Therefore, we made many figures for this proof, that would illustrate most of it. Here we give an explanation of how to interpret them. The figures typically represent several sets from $\hh$ and the relationship between them.

The elements of the ground set are represented by sectors in clockwise order, and sets are represented by colored ``cells'' between two consecutive circles. Thus, an element in a set is one colored cell. The sector of the element $x$ on Fig.~\ref{fig1} is marked by thick blue segments. The arcs with arrows indicate the size of a segment in the set. The length of intervals of length $1$ and $2$ are not marked. We always have $m=8$ on the figures.

In the charging-discharging part of the proof sets are marked in a different way:
\begin{itemize}\item the sets considered at that step and which have non-zero charge are marked with a color fill. We call these sets the \underline{current sets}.\item The sets that form a forbidden configuration with current sets are marked by dots. Such sets are not in the family.\item The pairs of sets that form a forbidden configuration together with the current set are marked by hatching. One of them must be missing from the family.\item The sets that are discharged on previous steps, or cannot be in the family because of the previous steps, are marked by stars.\end{itemize}}\vskip+0.1cm

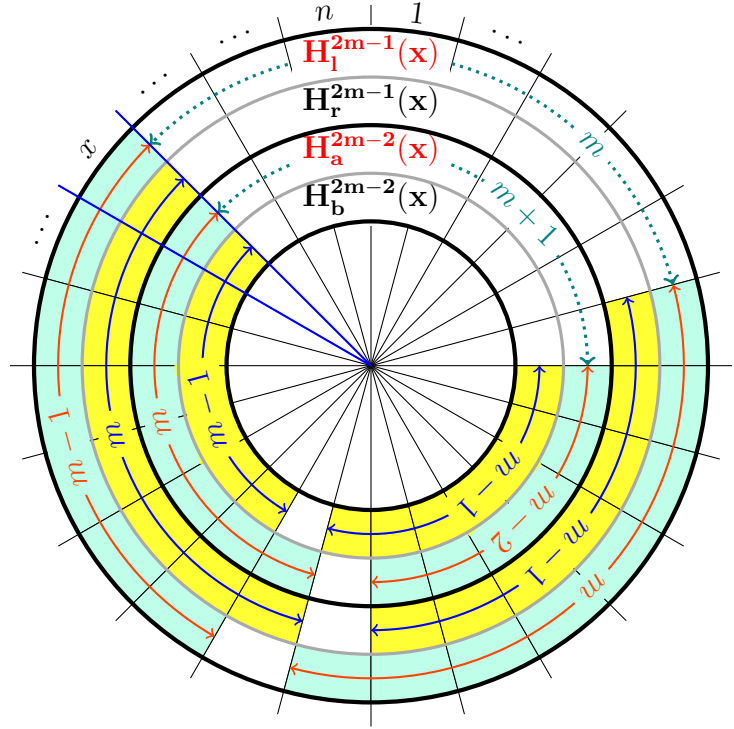
\begin{figure}\centering
\begin{tikzpicture}[scale=0.64]

\foreach \x in {1,...,24}
\draw (0,0) -- (15*\x:7.5);


\foreach \x in {0,9,10,11,12,13,14,15,17,18,19,20,21,22,23}
\filldraw[fill=Aquamarine!50, draw=black, opacity=1] (15*\x:6) -- (15*\x:7) arc (15*\x:15*(\x+1):7) -- (15*\x+15:6) arc (15*\x+15:15*\x:6);
\draw [->, OrangeRed,thick,domain=-45:15] plot ({6.5*cos(\x)}, {6.5*sin(\x)});
\draw [<-, OrangeRed,thick,domain=-105:-45] plot ({6.5*cos(\x)}, {6.5*sin(\x)})
node[anchor=center,fill=Aquamarine!50,  rotate=-135]{$m$};
\draw [->, OrangeRed,thick,domain=-165:-120] plot ({6.5*cos(\x)}, {6.5*sin(\x)});
\draw [<-, OrangeRed,thick,domain=-225:-165] plot ({6.5*cos(\x)}, {6.5*sin(\x)})
node[anchor=center,fill=Aquamarine!50,  rotate=-255]{$m-1$};

\draw [dotted,->, Teal,very thick,domain=45:135] plot ({6.5*cos(\x)}, {6.5*sin(\x)});
\draw [dotted, <-, Teal,very thick,domain=15:45] plot ({6.5*cos(\x)}, {6.5*sin(\x)})
node[anchor=center,fill=white,  rotate=-45]{$m$};
\node[fill=white] at (0,6.5) {$\textcolor{red}{\mathbf{H^{2m-1}_l(x)}}$};



\foreach \x in {9,10,11,12,13,14,15,16, 18,19,20,21,22,23,0}
\filldraw[fill=Yellow!80, draw=black, opacity=1] (15*\x:5) -- (15*\x:6) arc (15*\x:15*(\x+1):6) -- (15*\x+15:5) arc (15*\x+15:15*\x:5);
\draw [->, Blue,thick,domain=-165:-105] plot ({5.5*cos(\x)}, {5.5*sin(\x)});
\draw [<-, Blue,thick,domain=-225:-165] plot ({5.5*cos(\x)}, {5.5*sin(\x)})
node[anchor=center,fill=Yellow!80,  rotate=-255]{$m$};
\draw [->, Blue,thick,domain=-45:15] plot ({5.5*cos(\x)}, {5.5*sin(\x)});
\draw [<-, Blue,thick,domain=-90:-45] plot ({5.5*cos(\x)}, {5.5*sin(\x)})
node[anchor=center,fill=Yellow!80,  rotate=-135]{$m-1$};

\node[fill=white] at (0,5.5) {$\textcolor{black}{\mathbf{H^{2m-1}_r(x)}}$};



\foreach \x in {9,10,11,12,13,14,15,16, 18,19,20,21,22,23}
\filldraw[fill=Aquamarine!50, draw=black, opacity=1] (15*\x:4) -- (15*\x:5) arc (15*\x:15*(\x+1):5) -- (15*\x+15:4) arc (15*\x+15:15*\x:4);

\draw [->, OrangeRed,thick,domain=-165:-105] plot ({4.5*cos(\x)}, {4.5*sin(\x)});
\draw [<-, OrangeRed,thick,domain=-225:-165] plot ({4.5*cos(\x)}, {4.5*sin(\x)})
node[anchor=center,fill=Aquamarine!50,  rotate=-255]{$m$};
\draw [->, OrangeRed,thick,domain=-45:0] plot ({4.5*cos(\x)}, {4.5*sin(\x)});
\draw [<-, OrangeRed,thick,domain=-90:-45] plot ({4.5*cos(\x)}, {4.5*sin(\x)})
node[anchor=center,fill=Aquamarine!50,  rotate=-135]{$m-2$};

\draw [dotted,->, Teal,very thick,domain=45:135] plot ({4.5*cos(\x)}, {4.5*sin(\x)});
\draw [dotted, <-, Teal,very thick,domain=0:45] plot ({4.5*cos(\x)}, {4.5*sin(\x)})
node[anchor=center,fill=white,  rotate=-45]{$m+1$};

\node[fill=white] at (0,4.5) {$\textcolor{red}{\mathbf{H_a^{2m-2}(x)}}$};



\foreach \x in {9,10,11,12,13,14,15,17, 18,19,20,21,22,23}
\filldraw[fill=Yellow!80, draw=black, opacity=1] (15*\x:3) -- (15*\x:4) arc (15*\x:15*(\x+1):4) -- (15*\x+15:3) arc (15*\x+15:15*\x:3);

\draw [->, Blue,thick,domain=-165:-120] plot ({3.5*cos(\x)}, {3.5*sin(\x)});
\draw [<-, Blue,thick,domain=-225:-165] plot ({3.5*cos(\x)}, {3.5*sin(\x)})
node[anchor=center,fill=Yellow!80,  rotate=-255]{$m-1$};
\draw [->, Blue,thick,domain=-45:0] plot ({3.5*cos(\x)}, {3.5*sin(\x)});
\draw [<-, Blue,thick,domain=-105:-45] plot ({3.5*cos(\x)}, {3.5*sin(\x)})
node[anchor=center,fill=Yellow!80,  rotate=-135]{$m-1$};

\node[fill=white] at (0,3.5) {$\textcolor{black}{\mathbf{H^{2m-2}_b(x)}}$};


\foreach \x in {3,5,7}
\draw[ultra thick] (0,0) circle (\x);
\foreach \x in {4,6}
\draw[very thick,DarkGrey] (0,0) circle (\x);


\node[rotate=-7.5] at ({7.4*sin(7.5)},{cos(7.5)*7.4}) { $1$};
\node[rotate=7.5] at ({7.4*sin(-7.5)},{cos(7.5)*7.4}) { $n$};
\node[rotate=22.5] at ({7.4*sin(-22.5)},{cos(22.5)*7.4}) { $\ldots$};
\node[rotate=-22.5] at ({7.4*sin(22.5)},{cos(22.5)*7.4}) { $\ldots$};
\node[rotate=52.5] at ({7.4*sin(-52.5)},{cos(52.5)*7.4}) { $x$};
\node[rotate=37.5] at ({7.4*sin(-37.5)},{cos(37.5)*7.4}) { $\ldots$};
\node[rotate=67.5] at ({7.4*sin(-67.5)},{cos(67.5)*7.4}) { $\ldots$};

\foreach \x in {9,10}
\draw[thick, blue] (0,0) -- (15*\x:7.5);
\end{tikzpicture}
\caption{Non-interval sets from the family $\hh$ of sizes $2m-2$, $2m-1$.}\label{fig2}
\end{figure}

We note that $m+3<2m-2$ since $m\ge 6$, which guarantees that the sizes of the sets in the list do not coincide accidentally. As we have already said, the listed sets constitute the family $\hh$.  Note that the total weight of $j$-sets in $\hh$ sums up to $c_jn{n\choose j}$, where $c_{j}=1$ for $j\in I-\{2m-1\}$, $c_{m+2}=c_{2m-2}=\frac 12$, $c_{m+3}=\frac 13$, $c_{2m-1}=\frac 67$, and $c_j=0$ otherwise.

As in in the previous part of the theorem, we reduce the problem to the analysis of $\ff\cap \hh$ via the following claim.
\begin{cla}\label{clach2} To prove the inequality \eqref{eq7}, it is sufficient to show that for any choice of $\sigma$ we have
\begin{equation}\label{eq16}\sum_{F\in \mathcal F\cap \mathcal H}w(F)\le n\sum_{\substack{j\in [m+1,m+3]\cup\\ [2m-2,2m+1]}}c_j{n\choose j}.\end{equation}
Moreover, \eqref{eq16} implies that any partition-free family of maximal size contains all sets of sizes $k\in[m+2,2m-1]$.
\end{cla}
\begin{proof} As in the proof of Claim \ref{clach}, the equation \eqref{eq16} implies

$$\sum_{j\in[n]}nc_j\Big|\ff\cap{[n]\choose j}\Big| = \E\bigg[\sum_{F\in \mathcal F\cap \mathcal H}w(F)\bigg]\le n\sum_{\substack{j\in [m+1,m+3]\cup\\ [2m-2,2m+1]}}c_j{n\choose j}.$$
Recall that $c_k=0$ for $k$ that are not represented in $\hh$. Writing the last inequality in terms of $y_k$, we get that
\begin{equation}\label{eq4}\sum_{k\in [n]}c_ky_k\ge \sum_{\substack{j\in[m-2,m]\cup\\ \{2m+2,2m+3\}}}{n\choose j}.\end{equation}

Applying (\ref{eq1}) with $s_1 = 0,\ldots, m-4$, $s_2=m+2$, $s_3 = 2m-2-s_1$, we get that
\begin{multline*}\sum_{k=0}^{m-4}(y_k+y_{n-k})+\sum_{k=m+3}^{2m-3}\frac{{n\choose 2m-2-k} (y_k+y_{n-k})}{{n\choose k}}+\\ \frac{\Big(1+2{n\choose m-4}+\sum_{k=0}^{m-5}{n\choose k}\Big)(y_{m+2}+y_{2m-2})}{{n\choose m+2}} \ge 2\sum_{k=0}^{m-4}{n\choose k}.\end{multline*}
We have ${n\choose 2m-2-k}<\frac 14{n\choose k}$ for any $k=m+3,\ldots, 2m-3$, and $2{n\choose m-4}+1+\sum_{k=0}^{m-5}{n\choose k}\overset{(\ref{eq005})}{\le} 3{n\choose m-4}\le \frac 14{n\choose m+2}$. Therefore, we get that \begin{equation}\label{eq5}\sum_{k=0}^{m-4}(y_k+y_{n-k})+\frac 12\sum_{k=m+3}^{2m-3}y_k+\frac 14(y_{m+2}+y_{2m-2}) \ge 2\sum_{k=0}^{m-4}{n\choose k}.\end{equation}
We also use the following inequality, similar to \eqref{eq1}:
\begin{equation}\label{eq6}y_{m-3}+\frac 18(y_{m+2}+y_{2m-1})\ge y_{m-3}+\frac{{n\choose m-3}y_{m+2}}{{n\choose m+2}}+\frac{{n\choose m-3}y_{2m-1}}{{n\choose m+1}}\ge {n\choose m-3}.\end{equation}
(The choice of $\frac 18$ is somewhat arbitrary. We just need $8{3m\choose m-3}\le {3m\choose m+1}\le {3m\choose m+2}$, which is true for $m\ge 3$.)

Summing together \eqref{eq4}, \eqref{eq5}, and \eqref{eq6}, we see that coefficient in front of every $y_k$ is at most $1$, and we conclude that
$$\sum_{k\in[n]}c'_ky_k\ge {n\choose m}+{n\choose m-1}+2\sum_{j=0}^{m-2}{n\choose j},$$
where $0\le c'_k\le 1$ for any $k\in [n]$, and, moreover, $c'_k<1$ for $k\in [m+2,2m-1]$. This implies \eqref{eq7} along with the fact that in any partition-free family of maximal size all sets of sizes $k\in[m+2,2m-1]$ are present.
\end{proof}

We prove (\ref{eq16}) in the same way as \eqref{eq76}, but the discharging process will be different. Our goal is again to transfer all the charge from $(\le m)$- and $(\ge 2m+2)$-sets (\underline{outside layers sets}) in $\ff\cap \hh$ to the sets of size $m+1\le x \le 2m+1$ (\underline{middle layers sets}) in $\hh\setminus \ff$.\\

\textsc{Stage A. Transferring charge from $(\le(m-1))$-sets}.\\

\noindent\textbf{1. } Assume that for some $x\in [n]$ both $H^{m}(x)$ and $H^{m-2}(x-m-1)$ have non-zero charge. Then we transfer the charge of $H^{m-2}(x-m-1)$ to $H_a^{2m-2}(x)$, which is missing from $\ff$. We have $c(H^{m-2}(x-m-1)) ={n\choose m-2}<\frac 12{n\choose m+2} = w(H_a^{2m-2}(x))$. The set $H_a^{2m-2}(x)$ is not going to get any more charge. In what follows, we assume that there are no such pairs of $(m-2)$- and $m$-sets.
\begin{multicols}{2}
\begin{tikzpicture}[scale=0.5]

\foreach \x in {1,...,24}
\draw (0,0) -- (15*\x:6.5);


\foreach \x in {9,10,11,12,13,14,15,16}
\filldraw[fill=Aquamarine!50, draw=black, opacity=1] (15*\x:5) -- (15*\x:6) arc (15*\x:15*(\x+1):6) -- (15*\x+15:5) arc (15*\x+15:15*\x:5);
\node[fill=white] at (0,5.4) {{\scriptsize $\textcolor{JungleGreen}{H^{m}(x)}$}};



\foreach \x in {18,19,20,21,22,23}
\filldraw[fill=Yellow!80, draw=black, opacity=1] (15*\x:4) -- (15*\x:5) arc (15*\x:15*(\x+1):5) -- (15*\x+15:4) arc (15*\x+15:15*\x:4);
\node[fill=white] at (0,4.4) {{\scriptsize $\textcolor[rgb]{0.6,0.4,0}{H^{m-2}(x-m-1)}$}};




\foreach \x in {9,10,11,12,13,14,15,16, 18,19,20,21,22,23}
\filldraw[pattern=dots, pattern color=Red, draw=black] (15*\x:3) -- (15*\x:4) arc (15*\x:15*(\x+1):4) -- (15*\x+15:3) arc (15*\x+15:15*\x:3);

\node[fill=white] at (0,3.4) {{\scriptsize $\textcolor{red}{H_a^{2m-2}(x)}$}};


\foreach \x in {3,4,5,6}
\draw[very thick,DarkGrey] (0,0) circle (\x);

\node at (-5.5,5.5) {$\mathbf{1.}$};
\node[rotate=52.5] at ({6.4*sin(-52.5)},{cos(52.5)*6.4}) {{\scriptsize $x$}};
\node[rotate=7.5] at ({6.4*sin(172.5)},{cos(172.5)*6.4}) {{\scriptsize $x-m-1$}};

\end{tikzpicture}
\begin{tikzpicture}[scale=0.5]

\foreach \x in {1,...,24}
\draw (0,0) -- (15*\x:7.5);


\foreach \x in {18,19,20,21,22,23}
\filldraw[fill=Yellow!80, draw=black, opacity=1] (15*\x:5) -- (15*\x:6) arc (15*\x:15*(\x+1):6) -- (15*\x+15:5) arc (15*\x+15:15*\x:5);
\node[fill=white] at (0,5.4) {{\scriptsize $\textcolor[rgb]{0.6,0.4,0}{H^{m-2}(x-m-1)}$}};




\foreach \x in {0,9,10,11,12,13,14,15,16,17}
\filldraw[pattern=north west lines, pattern color=JungleGreen, draw=black, opacity=1] (15*\x:4) -- (15*\x:5) arc (15*\x:15*(\x+1):5) -- (15*\x+15:4) arc (15*\x+15:15*\x:4);
\node[fill=white] at (0,4.4) {{\scriptsize $\textcolor{JungleGreen}{H_l^{m+2}(x+m+1)}$}};



\foreach \x in {0, 9,10,11,12,13,14,15,16,17, 18,19,20,21,22,23}
\filldraw[pattern=north west lines, pattern color=blue, draw=black] (15*\x:3) -- (15*\x:4) arc (15*\x:15*(\x+1):4) -- (15*\x+15:3) arc (15*\x+15:15*\x:3);

\node[fill=white] at (0,3.4) {{\scriptsize $\textcolor{blue}{H^{2m}(x)}$}};



\foreach \x in {9,...,16}
\filldraw[pattern=fivepointed stars, pattern color=red, draw=black, opacity=1] (15*\x:6) -- (15*\x:7) arc (15*\x:15*(\x+1):7) -- (15*\x+15:6) arc (15*\x+15:15*\x:6);
\node[fill=white] at (0,6.4) {{\scriptsize $\textcolor{red}{H^{m}(x)}$}};



\foreach \x in {3,4,5,6,7}
\draw[very thick,DarkGrey] (0,0) circle (\x);


\node at (-6.5,6.5) {$\mathbf{2.}$};

\node[rotate=52.5] at ({7.4*sin(-52.5)},{cos(52.5)*7.4}) {{\scriptsize $x$}};
\node[rotate=-82.5] at ({7.4*sin(82.5)},{cos(82.5)*7.4}) {{\scriptsize $x+m+1$}};
\node[rotate=7.5] at ({7.4*sin(172.5)},{cos(172.5)*7.4}) {{\scriptsize $x-m-1$}};

\end{tikzpicture}
\end{multicols}

\noindent\textbf{2. } Next, assume that for some $x$ the set $H^{m-2}(x-m-1)$ still has non-zero charge. (Note that $H^{m+2}_l(x+m+1)\cup H^{m-2}(x-m-1)=H^{2m}(x)$. Also note that $H^{m}(x)$ is not in the family due to step 1, it is marked accordingly on the figure.) Thus at least one of the two sets $H^{m+2}_l(x+m+1),H^{2m}(x)$ has zero charge, and we transfer the ${n\choose m-2}$ charge to this set.  We have ${n\choose m-2}\le w(H^{2m}(x))$ and ${n\choose m-2}=\frac{(m-1)(m+2)}{4(2m+1)(2m-1)}{n\choose m+2}=\frac 1{12}{n\choose m+2}< \frac 12 w\big(H^{m+2}_l(x+m+1)\big)$. From now on we assume that there are no $(m-2)$-sets with non-zero charge.\\

\noindent\textbf{3. } Assume that for some $x\in[n]$ both $H^{m-1}(x)$ and $H^{m-1}(x-m)$ have non-zero charge. Then the set $H_b^{2m-2}(x)$ is missing from $\ff$, and we transfer all the charge from $H^{m-1}(x-m)$ to it. We have ${n\choose m-1}= \frac {(m+2)(m+1)m}{(2m-1)2m(2m+1)}{n\choose m+2}<\frac 14{n\choose m+2} = w(H^{2m-2}_b(x))$. From now on we assume that there are no such pairs of charged $(m-1)$-sets.

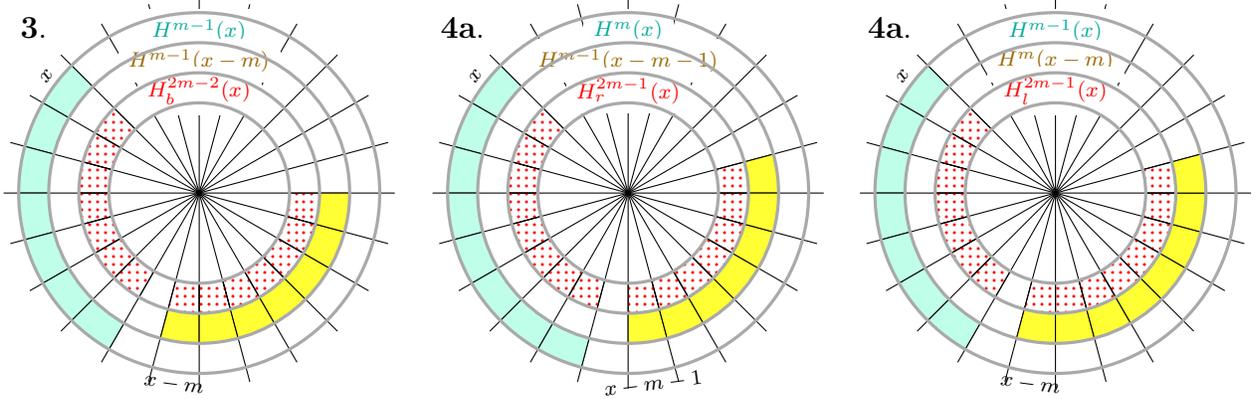
\begin{vwcol}[widths={0.33,0.33,0.33}, rule=0pt]
\begin{tikzpicture}[scale=0.4]

\foreach \x in {1,...,24}
\draw (0,0) -- (15*\x:6.5);


\foreach \x in {9,10,11,12,13,14,15}
\filldraw[fill=Aquamarine!50, draw=black, opacity=1] (15*\x:5) -- (15*\x:6) arc (15*\x:15*(\x+1):6) -- (15*\x+15:5) arc (15*\x+15:15*\x:5);
\node[fill=white] at (0,5.4) {{\scriptsize $\textcolor{JungleGreen}{H^{m-1}(x)}$}};



\foreach \x in {17,18,19,20,21,22,23}
\filldraw[fill=Yellow!80, draw=black, opacity=1] (15*\x:4) -- (15*\x:5) arc (15*\x:15*(\x+1):5) -- (15*\x+15:4) arc (15*\x+15:15*\x:4);
\node[fill=white] at (0,4.4) {{\scriptsize $\textcolor[rgb]{0.6,0.4,0}{H^{m-1}(x-m)}$}};




\foreach \x in {9,10,11,12,13,14,15,17, 18,19,20,21,22,23}
\filldraw[pattern=dots, pattern color=Red, draw=black] (15*\x:3) -- (15*\x:4) arc (15*\x:15*(\x+1):4) -- (15*\x+15:3) arc (15*\x+15:15*\x:3);

\node[fill=white] at (0,3.4) {{\scriptsize $\textcolor{red}{H_b^{2m-2}(x)}$}};


\foreach \x in {3,4,5,6}
\draw[very thick,DarkGrey] (0,0) circle (\x);

\node at (-5.5,5.5) {$\mathbf{3.}$};
\node[rotate=52.5] at ({6.4*sin(-52.5)},{cos(52.5)*6.4}) {{\scriptsize $x$}};
\node[rotate=-7.5] at ({6.4*sin(-172.5)},{cos(-172.5)*6.4}) {{\scriptsize $x-m$}};

\end{tikzpicture}
\begin{tikzpicture}[scale=0.4]

\foreach \x in {1,...,24}
\draw (0,0) -- (15*\x:6.5);


\foreach \x in {9,...,16}
\filldraw[fill=Aquamarine!50, draw=black, opacity=1] (15*\x:5) -- (15*\x:6) arc (15*\x:15*(\x+1):6) -- (15*\x+15:5) arc (15*\x+15:15*\x:5);
\node[fill=white] at (0,5.4) {{\scriptsize $\textcolor{JungleGreen}{H^{m}(x)}$}};



\foreach \x in {0,18,19,20,21,22,23}
\filldraw[fill=Yellow!80, draw=black, opacity=1] (15*\x:4) -- (15*\x:5) arc (15*\x:15*(\x+1):5) -- (15*\x+15:4) arc (15*\x+15:15*\x:4);
\node[fill=white] at (0,4.4) {{\scriptsize $\textcolor[rgb]{0.6,0.4,0}{H^{m-1}(x-m-1)}$}};




\foreach \x in {9,10,11,12,13,14,15,16, 18,19,20,21,22,23,0}
\filldraw[pattern=dots, pattern color=Red, draw=black] (15*\x:3) -- (15*\x:4) arc (15*\x:15*(\x+1):4) -- (15*\x+15:3) arc (15*\x+15:15*\x:3);

\node[fill=white] at (0,3.4) {{\scriptsize $\textcolor{red}{H_r^{2m-1}(x)}$}};


\foreach \x in {3,4,5,6}
\draw[very thick,DarkGrey] (0,0) circle (\x);

\node at (-5.5,5.5) {$\mathbf{4a.}$};
\node[rotate=52.5] at ({6.4*sin(-52.5)},{cos(52.5)*6.4}) {{\scriptsize $x$}};
\node[rotate=7.5] at ({6.4*sin(172.5)},{cos(172.5)*6.4}) {{\scriptsize $x-m-1$}};

\end{tikzpicture}
\begin{tikzpicture}[scale=0.4]

\foreach \x in {1,...,24}
\draw (0,0) -- (15*\x:6.5);


\foreach \x in {9,10,11,12,13,14,15}
\filldraw[fill=Aquamarine!50, draw=black, opacity=1] (15*\x:5) -- (15*\x:6) arc (15*\x:15*(\x+1):6) -- (15*\x+15:5) arc (15*\x+15:15*\x:5);
\node[fill=white] at (0,5.4) {{\scriptsize $\textcolor{JungleGreen}{H^{m-1}(x)}$}};



\foreach \x in {17,18,19,20,21,22,23,0}
\filldraw[fill=Yellow!80, draw=black, opacity=1] (15*\x:4) -- (15*\x:5) arc (15*\x:15*(\x+1):5) -- (15*\x+15:4) arc (15*\x+15:15*\x:4);
\node[fill=white] at (0,4.4) {{\scriptsize $\textcolor[rgb]{0.6,0.4,0}{H^{m}(x-m)}$}};




\foreach \x in {9,10,11,12,13,14,15,17, 18,19,20,21,22,23,0}
\filldraw[pattern=dots, pattern color=Red, draw=black] (15*\x:3) -- (15*\x:4) arc (15*\x:15*(\x+1):4) -- (15*\x+15:3) arc (15*\x+15:15*\x:3);

\node[fill=white] at (0,3.4) {{\scriptsize $\textcolor{red}{H_l^{2m-1}(x)}$}};


\foreach \x in {3,4,5,6}
\draw[very thick,DarkGrey] (0,0) circle (\x);

\node at (-5.5,5.5) {$\mathbf{4a.}$};
\node[rotate=52.5] at ({6.4*sin(-52.5)},{cos(52.5)*6.4}) {{\scriptsize $x$}};
\node[rotate=-7.5] at ({6.4*sin(-172.5)},{cos(-172.5)*6.4}) {{\scriptsize $x-m$}};

\end{tikzpicture}
\end{vwcol}

\noindent\textbf{4a. } Assume that for some $x\in[n]$ both $H^{m}(x)$ and $H^{m-1}(x-m-1)$ have non-zero charge. Then the set $H_r^{2m-1}(x)$ is missing from $\ff$, and we transfer all the charge from $H^{m-1}(x-m-1)$ to it. We have ${n\choose m-1}= \frac {m+1}{4m+2}{n\choose m+1}<\frac 27{n\choose m+1} = w(H^{2m-1}_r(x))$. Analogously, if both $H^{m-1}(x)$ and $H^{m}(x-m)$ have non-zero charge, then we transfer the charge of $H^{m-1}(x)$ to the missing $H_l^{2m-1}(x)$. The calculations are the same. From now on we assume that there are no such pairs of charged $(m-1)$-sets and $m$-sets.

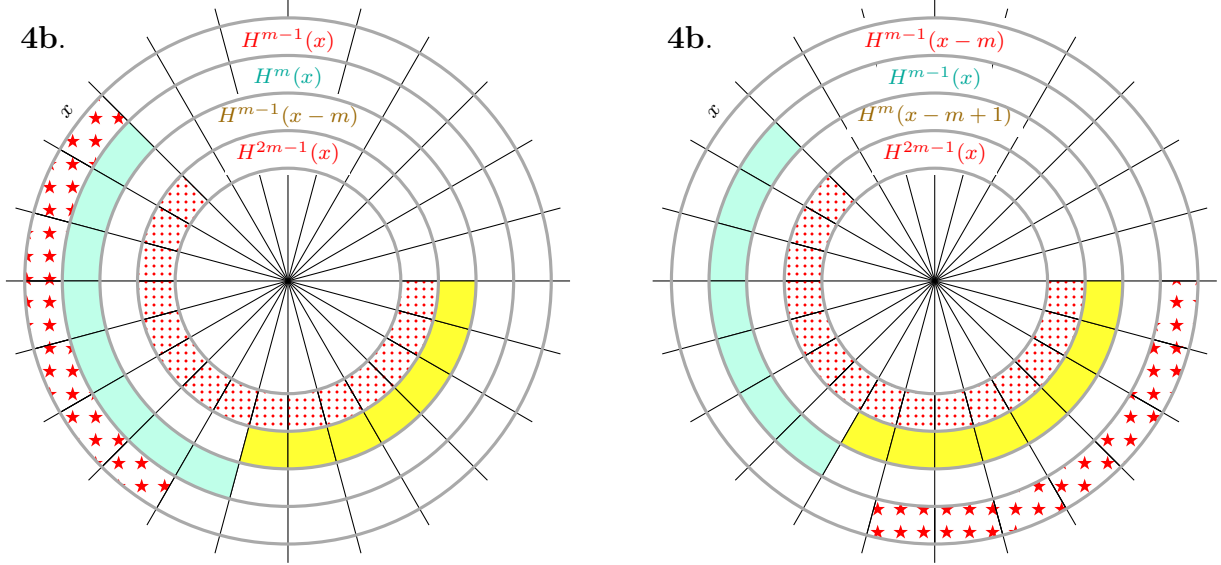
\begin{vwcol}[widths={0.5,0.5}, justify=center, rule=0pt]
\begin{tikzpicture}[scale=0.5]

\foreach \x in {1,...,24}
\draw (0,0) -- (15*\x:7.5);


\foreach \x in {9,...,16}
\filldraw[fill=Aquamarine!50, draw=black, opacity=1] (15*\x:5) -- (15*\x:6) arc (15*\x:15*(\x+1):6) -- (15*\x+15:5) arc (15*\x+15:15*\x:5);
\node[fill=white] at (0,5.4) {{\scriptsize $\textcolor{JungleGreen}{H^{m}(x)}$}};



\foreach \x in {17,18,19,20,21,22,23}
\filldraw[fill=Yellow!80, draw=black, opacity=1] (15*\x:4) -- (15*\x:5) arc (15*\x:15*(\x+1):5) -- (15*\x+15:4) arc (15*\x+15:15*\x:4);
\node[fill=white] at (0,4.4) {{\scriptsize $\textcolor[rgb]{0.6,0.4,0}{H^{m-1}(x-m)}$}};




\foreach \x in {9,10,11,12,13,14,15,16,17, 18,19,20,21,22,23}
\filldraw[pattern=dots, pattern color=Red, draw=black] (15*\x:3) -- (15*\x:4) arc (15*\x:15*(\x+1):4) -- (15*\x+15:3) arc (15*\x+15:15*\x:3);

\node[fill=white] at (0,3.4) {{\scriptsize $\textcolor{red}{H^{2m-1}(x)}$}};



\foreach \x in {9,...,15}
\filldraw[pattern=fivepointed stars, pattern color=red, draw=black, opacity=1] (15*\x:6) -- (15*\x:7) arc (15*\x:15*(\x+1):7) -- (15*\x+15:6) arc (15*\x+15:15*\x:6);
\node[fill=white] at (0,6.4) {{\scriptsize $\textcolor{red}{H^{m-1}(x)}$}};



\foreach \x in {3,4,5,6,7}
\draw[very thick,DarkGrey] (0,0) circle (\x);

\node at (-6.5,6.5) {$\mathbf{4b.}$};
\node[rotate=52.5] at ({7.4*sin(-52.5)},{cos(52.5)*7.4}) {{\scriptsize $x$}};

\end{tikzpicture}
\begin{tikzpicture}[scale=0.5]

\foreach \x in {1,...,24}
\draw (0,0) -- (15*\x:7.5);


\foreach \x in {9,10,11,12,13,14,15}
\filldraw[fill=Aquamarine!50, draw=black, opacity=1] (15*\x:5) -- (15*\x:6) arc (15*\x:15*(\x+1):6) -- (15*\x+15:5) arc (15*\x+15:15*\x:5);
\node[fill=white] at (0,5.4) {{\scriptsize $\textcolor{JungleGreen}{H^{m-1}(x)}$}};



\foreach \x in {16,17,18,19,20,21,22,23}
\filldraw[fill=Yellow!80, draw=black, opacity=1] (15*\x:4) -- (15*\x:5) arc (15*\x:15*(\x+1):5) -- (15*\x+15:4) arc (15*\x+15:15*\x:4);
\node[fill=white] at (0,4.4) {{\scriptsize $\textcolor[rgb]{0.6,0.4,0}{H^{m}(x-m+1)}$}};




\foreach \x in {9,10,11,12,13,14,15,16,17, 18,19,20,21,22,23}
\filldraw[pattern=dots, pattern color=Red, draw=black] (15*\x:3) -- (15*\x:4) arc (15*\x:15*(\x+1):4) -- (15*\x+15:3) arc (15*\x+15:15*\x:3);

\node[fill=white] at (0,3.4) {{\scriptsize $\textcolor{red}{H^{2m-1}(x)}$}};



\foreach \x in {17,...,23}
\filldraw[pattern=fivepointed stars, pattern color=red, draw=black, opacity=1] (15*\x:6) -- (15*\x:7) arc (15*\x:15*(\x+1):7) -- (15*\x+15:6) arc (15*\x+15:15*\x:6);
\node[fill=white] at (0,6.4) {{\scriptsize $\textcolor{red}{H^{m-1}(x-m)}$}};



\foreach \x in {3,4,5,6,7}
\draw[very thick,DarkGrey] (0,0) circle (\x);

\node at (-6.5,6.5) {$\mathbf{4b.}$};
\node[rotate=52.5] at ({7.4*sin(-52.5)},{cos(52.5)*7.4}) {{\scriptsize $x$}};

\end{tikzpicture}
\end{vwcol}

\noindent\textbf{4b. } Assume that for some $x$ either both $H^{m}(x)$ and $H^{m-1}(x-m)$ have non-zero charge, or both $H^{m-1}(x)$ and $H^{m}(x-m+1)$ have non-zero charge. (Note that due to step 3, both possibilities cannot happen at the same time.) Then the set $H^{2m-1}(x)$ is missing from $\ff$, and we transfer all the charge from the $(m-1)$-set ($H^{m-1}(x)$ or $H^{m-1}(x-m+1)$) to it. The calculations are the same as in the step 4a.\\

\noindent\textbf{5. } Assume that  for some $x\in[n]$ the set $H^{m-1}(x)$ still has non-zero charge. Then one of the sets in each pair $(H^{2m}(x+1), H^{m+1}_l(x+1))$, $(H^{2m}(x+m),H^{m+1}_r(x+m))$ is missing from the family. We transfer half of the charge of $H^{m-1}(x)$ to each of those two sets. The charge a $2m$-set could get at this step is ${n\choose m-1}$, which together with the charge that a missing $2m$-set could accumulate on step 2 gives at most ${n\choose m-1}+{n\choose m-2}<{n\choose m}=w(H^{2m}(x))$. No $2m$-set that got some charge on steps 2 or 5 is going to get any more charge. See on the figure that some $m$-sets are forbidden due to steps 4a and 4b.
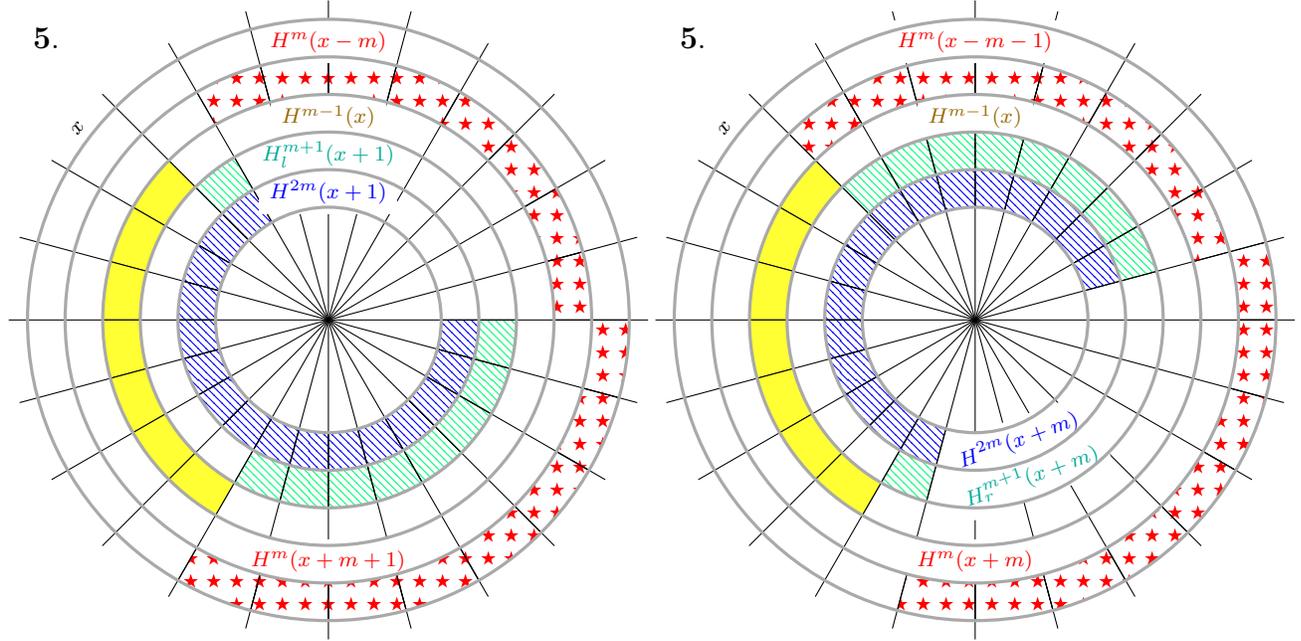
\begin{vwcol}[widths={0.5,0.5}, justify=flush, rule=0pt]
\begin{tikzpicture}[scale=0.5]

\foreach \x in {1,...,24}
\draw (0,0) -- (15*\x:8.5);


\foreach \x in {9,10,11,12,13,14,15}
\filldraw[fill=Yellow!80, draw=black, opacity=1] (15*\x:5) -- (15*\x:6) arc (15*\x:15*(\x+1):6) -- (15*\x+15:5) arc (15*\x+15:15*\x:5);
\node[fill=white] at (0,5.4) {{\scriptsize $\textcolor[rgb]{0.6,0.4,0}{H^{m-1}(x)}$}};



\foreach \x in {8,16,17,18,19,20,21,22,23}
\filldraw[pattern=north west lines, pattern color=JungleGreen, draw=black, opacity=1] (15*\x:4) -- (15*\x:5) arc (15*\x:15*(\x+1):5) -- (15*\x+15:4) arc (15*\x+15:15*\x:4);
\node[fill=white] at (0,4.4) {{\scriptsize $\textcolor{JungleGreen}{H_l^{m+1}(x+1)}$}};




\foreach \x in {8,9,10,11,12,13,14,15,16,17,18,19,20,21,22,23}
\filldraw[pattern=north west lines, pattern color=blue, draw=black] (15*\x:3) -- (15*\x:4) arc (15*\x:15*(\x+1):4) -- (15*\x+15:3) arc (15*\x+15:15*\x:3);

\node[fill=white] at (0,3.4) {{\scriptsize $\textcolor{blue}{H^{2m}(x+1)}$}};



\foreach \x in {0,...,7}
\filldraw[pattern=fivepointed stars, pattern color=red, draw=black, opacity=1] (15*\x:6) -- (15*\x:7) arc (15*\x:15*(\x+1):7) -- (15*\x+15:6) arc (15*\x+15:15*\x:6);
\node[fill=white] at (0,-6.4) {{\scriptsize $\textcolor{red}{H^{m}(x+m+1)}$}};




\foreach \x in {16,...,23}
\filldraw[pattern=fivepointed stars, pattern color=red, draw=black, opacity=1] (15*\x:7) -- (15*\x:8) arc (15*\x:15*(\x+1):8) -- (15*\x+15:7) arc (15*\x+15:15*\x:7);
\node[fill=white] at (0,7.4) {{\scriptsize $\textcolor{red}{H^{m}(x-m)}$}};



\foreach \x in {3,4,5,6,7,8}
\draw[very thick,DarkGrey] (0,0) circle (\x);

\node at (-7.5,7.5) {$\mathbf{5.}$};

\node[rotate=52.5] at ({8.4*sin(-52.5)},{cos(52.5)*8.4}) {{\scriptsize $x$}};

\end{tikzpicture}
\begin{tikzpicture}[scale=0.5]

\foreach \x in {1,...,24}
\draw (0,0) -- (15*\x:8.5);


\foreach \x in {9,10,11,12,13,14,15}
\filldraw[fill=Yellow!80, draw=black, opacity=1] (15*\x:5) -- (15*\x:6) arc (15*\x:15*(\x+1):6) -- (15*\x+15:5) arc (15*\x+15:15*\x:5);
\node[fill=white] at (0,5.4) {{\scriptsize $\textcolor[rgb]{0.6,0.4,0}{H^{m-1}(x)}$}};



\foreach \x in {1,2,3,4,5,6,7,8,16}
\filldraw[pattern=north west lines, pattern color=JungleGreen, draw=black, opacity=1] (15*\x:4) -- (15*\x:5) arc (15*\x:15*(\x+1):5) -- (15*\x+15:4) arc (15*\x+15:15*\x:4);
\node[rotate=20][fill=white] at ({4.4*sin(20)},{-4.4*cos(20)}) {{\scriptsize $\textcolor{JungleGreen}{H_r^{m+1}(x+m)}$}};




\foreach \x in {1,...,16}
\filldraw[pattern=north west lines, pattern color=blue, draw=black] (15*\x:3) -- (15*\x:4) arc (15*\x:15*(\x+1):4) -- (15*\x+15:3) arc (15*\x+15:15*\x:3);

\node[rotate=20][fill=white] at ({3.4*sin(20)},{-3.4*cos(20)}) {{\scriptsize $\textcolor{blue}{H^{2m}(x+m)}$}};



\foreach \x in {1,...,8}
\filldraw[pattern=fivepointed stars, pattern color=red, draw=black, opacity=1] (15*\x:6) -- (15*\x:7) arc (15*\x:15*\x+15:7) -- (15*\x+15:6) arc (15*\x+15:15*\x:6);
\node[fill=white] at (0,-6.4) {{\scriptsize $\textcolor{red}{H^{m}(x+m)}$}};




\foreach \x in {0,17,18,19,20,21,22,23}
\filldraw[pattern=fivepointed stars, pattern color=red, draw=black, opacity=1] (15*\x:7) -- (15*\x:8) arc (15*\x:15*\x+15:8) -- (15*\x+15:7) arc (15*\x+15:15*\x:7);
\node[fill=white] at (0,7.4) {{\scriptsize $\textcolor{red}{H^{m}(x-m-1)}$}};



\foreach \x in {3,4,5,6,7,8}
\draw[very thick,DarkGrey] (0,0) circle (\x);

\node at (-7.5,7.5) {$\mathbf{5.}$};

\node[rotate=52.5] at ({8.4*sin(-52.5)},{cos(52.5)*8.4}) {{\scriptsize $x$}};

\end{tikzpicture}
\end{vwcol}

At this point all the $(m-2)$- and $(m-1)$-sets are discharged. We remark that $(2m-2)$- and $(2m-1)$-sets are not going to get any more charge.\\


\textsc{Stage B. Transferring charge from $(\ge 2m+2)$-sets}.\\

\noindent\textbf{6. } Assume that for some $x$ and $\mathbf{j\in\{2,3\}}$ both $H^{2m+j}(x)$ and $H^{m}(x-1)$ are in the family. Then $H^{m+j}_l(x)$ is not, and we transfer the charge of $H^{2m+j}(x)$ to $H^{m+j}_l(x)$. It gets ${n\choose m-j} \le \frac 1{12}{n\choose m+j}=\frac 12w(H^{m+j}_l(x))$ charge. Similarly, if for some $x$ and $j\in\{2,3\}$ both $H^{2m+j}(x)$ and $H^{m}(x-m-j+1)$ are in the family, then $H^{m+j}_r(x)$ is not. We transfer the charge of $H^{2m+j}(x)$ to $H^{m+j}_r(x)$. The calculations stay the same. Note that for both $(m+3)$-sets and $(m+2)$-sets the charge received until now does not surpass their capacity. From now on we assume that there are no such pairs of $m$-sets and $(2m+j)$-sets, where both sets are charged. \\

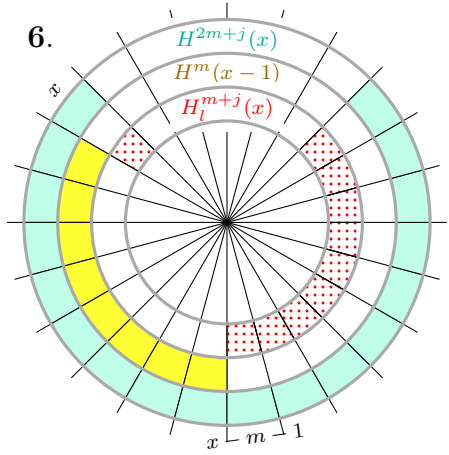
\begin{vwcol}[widths={0.5,0.5}, justify=center, rule=0pt]
\begin{tikzpicture}[scale=0.45]

\foreach \x in {1,...,24}
\draw (0,0) -- (15*\x:6.5);


\foreach \x in {9,10,11,12,13,14,15,16,17,18,19,20,21,22,23,0,1,2}
\filldraw[fill=Aquamarine!50, draw=black, opacity=1] (15*\x:5) -- (15*\x:6) arc (15*\x:15*(\x+1):6) -- (15*\x+15:5) arc (15*\x+15:15*\x:5);
\node[fill=white] at (0,5.4) {{\scriptsize $\textcolor{JungleGreen}{H^{2m+j}(x)}$}};



\foreach \x in {10,...,17}
\filldraw[fill=Yellow!80, draw=black, opacity=1] (15*\x:4) -- (15*\x:5) arc (15*\x:15*(\x+1):5) -- (15*\x+15:4) arc (15*\x+15:15*\x:4);
\node[fill=white] at (0,4.4) {{\scriptsize $\textcolor[rgb]{0.6,0.4,0}{H^{m}(x-1)}$}};




\foreach \x in {0,1,2,9,18,19,20,21,22,23}
\filldraw[pattern=dots, pattern color=Red, draw=black] (15*\x:3) -- (15*\x:4) arc (15*\x:15*(\x+1):4) -- (15*\x+15:3) arc (15*\x+15:15*\x:3);

\node[fill=white] at (0,3.4) {{\scriptsize $\textcolor{red}{H_l^{m+j}(x)}$}};


\foreach \x in {3,4,5,6}
\draw[very thick,DarkGrey] (0,0) circle (\x);

\node at (-5.5,5.5) {$\mathbf{6.}$};
\node[rotate=52.5] at ({6.4*sin(-52.5)},{cos(52.5)*6.4}) {{\scriptsize $x$}};
\node[rotate=7.5] at ({6.4*sin(172.5)},{cos(172.5)*6.4}) {{\scriptsize $x-m-1$}};

\end{tikzpicture}
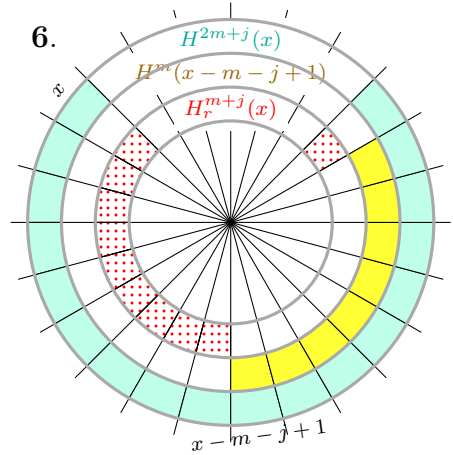
\begin{tikzpicture}[scale=0.45]

\foreach \x in {1,...,24}
\draw (0,0) -- (15*\x:6.5);


\foreach \x in {9,10,11,12,13,14,15,16,17,18,19,20,21,22,23,0,1,2}
\filldraw[fill=Aquamarine!50, draw=black, opacity=1] (15*\x:5) -- (15*\x:6) arc (15*\x:15*(\x+1):6) -- (15*\x+15:5) arc (15*\x+15:15*\x:5);
\node[fill=white] at (0,5.4) {{\scriptsize $\textcolor{JungleGreen}{H^{2m+j}(x)}$}};



\foreach \x in {18,19,20,21,22,23,0,1}
\filldraw[fill=Yellow!80, draw=black, opacity=1] (15*\x:4) -- (15*\x:5) arc (15*\x:15*(\x+1):5) -- (15*\x+15:4) arc (15*\x+15:15*\x:4);
\node[fill=white] at (0,4.4) {{\scriptsize $\textcolor[rgb]{0.6,0.4,0}{H^{m}(x-m-j+1)}$}};




\foreach \x in {2,9,10,11,12,13,14,15,16,17}
\filldraw[pattern=dots, pattern color=Red, draw=black] (15*\x:3) -- (15*\x:4) arc (15*\x:15*(\x+1):4) -- (15*\x+15:3) arc (15*\x+15:15*\x:3);

\node[fill=white] at (0,3.4) {{\scriptsize $\textcolor{red}{H_r^{m+j}(x)}$}};


\foreach \x in {3,4,5,6}
\draw[very thick,DarkGrey] (0,0) circle (\x);

\node at (-5.5,5.5) {$\mathbf{6.}$};
\node[rotate=52.5] at ({6.4*sin(-52.5)},{cos(52.5)*6.4}) {{\scriptsize $x$}};
\node[rotate=7.5] at ({6.4*sin(172.5)},{cos(172.5)*6.4}) {{\scriptsize $x-m-j+1$}};

\end{tikzpicture}
\end{vwcol}

\noindent\textbf{7. } If there remains a $(2m+2)$-set $H^{2m+2}(x)$ with non-zero charge, then in the pair $H^{m+1}(x),$ $H^{m+1}(x-m-1)$ one set is missing. We transfer the charge of this $(2m+2)$-set to  the missing set.

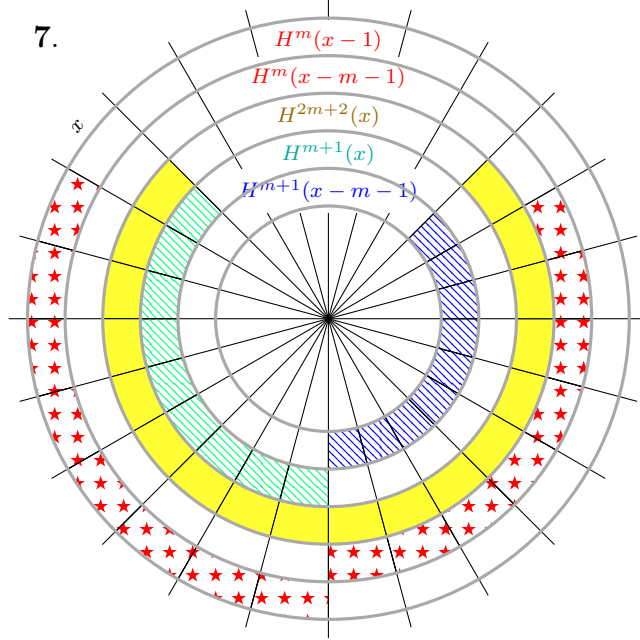
\begin{vwcol}[widths={0.5,0.5}, justify=flush, rule=0pt]
\begin{tikzpicture}[scale=0.5]

\foreach \x in {1,...,24}
\draw (0,0) -- (15*\x:8.5);


\foreach \x in {9,10,11,12,13,14,15,16,17,18,19,20,21,22,23,0,1,2}
\filldraw[fill=Yellow!80, draw=black, opacity=1] (15*\x:5) -- (15*\x:6) arc (15*\x:15*(\x+1):6) -- (15*\x+15:5) arc (15*\x+15:15*\x:5);
\node[fill=white] at (0,5.4) {{\scriptsize $\textcolor[rgb]{0.6,0.4,0}{H^{2m+2}(x)}$}};



\foreach \x in {9,...,17}
\filldraw[pattern=north west lines, pattern color=JungleGreen, draw=black, opacity=1] (15*\x:4) -- (15*\x:5) arc (15*\x:15*(\x+1):5) -- (15*\x+15:4) arc (15*\x+15:15*\x:4);
\node[fill=white] at (0,4.4) {{\scriptsize $\textcolor{JungleGreen}{H^{m+1}(x)}$}};




\foreach \x in {18,19,20,21,22,23,0,1,2}
\filldraw[pattern=north west lines, pattern color=blue, draw=black] (15*\x:3) -- (15*\x:4) arc (15*\x:15*(\x+1):4) -- (15*\x+15:3) arc (15*\x+15:15*\x:3);

\node[fill=white] at (0,3.4) {{\scriptsize $\textcolor{blue}{H^{m+1}(x-m-1)}$}};



\foreach \x in {0,1,18,19,20,21,22,23}
\filldraw[pattern=fivepointed stars, pattern color=red, draw=black, opacity=1] (15*\x:6) -- (15*\x:7) arc (15*\x:15*(\x+1):7) -- (15*\x+15:6) arc (15*\x+15:15*\x:6);
\node[fill=white] at (0,6.4) {{\scriptsize $\textcolor{red}{H^{m}(x-m-1)}$}};




\foreach \x in {10,...,17}
\filldraw[pattern=fivepointed stars, pattern color=red, draw=black, opacity=1] (15*\x:7) -- (15*\x:8) arc (15*\x:15*(\x+1):8) -- (15*\x+15:7) arc (15*\x+15:15*\x:7);
\node[fill=white] at (0,7.4) {{\scriptsize $\textcolor{red}{H^{m}(x-1)}$}};



\foreach \x in {3,4,5,6,7,8}
\draw[very thick,DarkGrey] (0,0) circle (\x);

\node at (-7.5,7.5) {$\mathbf{7.}$};

\node[rotate=52.5] at ({8.4*sin(-52.5)},{cos(52.5)*8.4}) {{\scriptsize $x$}};

\end{tikzpicture}
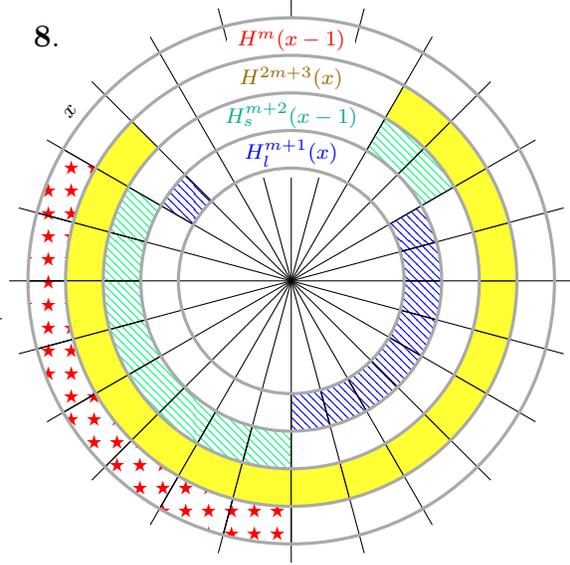
\begin{tikzpicture}[scale=0.5]

\foreach \x in {1,...,24}
\draw (0,0) -- (15*\x:7.5);


\foreach \x in {9,10,11,12,13,14,15,16,17,18,19,20,21,22,23,0,1,2,3}
\filldraw[fill=Yellow!80, draw=black, opacity=1] (15*\x:5) -- (15*\x:6) arc (15*\x:15*(\x+1):6) -- (15*\x+15:5) arc (15*\x+15:15*\x:5);
\node[fill=white] at (0,5.4) {{\scriptsize $\textcolor[rgb]{0.6,0.4,0}{H^{2m+3}(x)}$}};



\foreach \x in {2,3,10,11,12,13,14,15,16,17}
\filldraw[pattern=north west lines, pattern color=JungleGreen, draw=black, opacity=1] (15*\x:4) -- (15*\x:5) arc (15*\x:15*(\x+1):5) -- (15*\x+15:4) arc (15*\x+15:15*\x:4);
\node[fill=white] at (0,4.4) {{\scriptsize $\textcolor{JungleGreen}{H_s^{m+2}(x-1)}$}};




\foreach \x in {9,18,19,20,21,22,23,0,1}
\filldraw[pattern=north west lines, pattern color=blue, draw=black] (15*\x:3) -- (15*\x:4) arc (15*\x:15*(\x+1):4) -- (15*\x+15:3) arc (15*\x+15:15*\x:3);

\node[fill=white] at (0,3.4) {{\scriptsize $\textcolor{blue}{H_l^{m+1}(x)}$}};



\foreach \x in {10,...,17}
\filldraw[pattern=fivepointed stars, pattern color=red, draw=black, opacity=1] (15*\x:6) -- (15*\x:7) arc (15*\x:15*(\x+1):7) -- (15*\x+15:6) arc (15*\x+15:15*\x:6);
\node[fill=white] at (0,6.4) {{\scriptsize $\textcolor{red}{H^{m}(x-1)}$}};



\foreach \x in {3,4,5,6,7}
\draw[very thick,DarkGrey] (0,0) circle (\x);

\node at (-6.5,6.5) {$\mathbf{8.}$};

\node[rotate=52.5] at ({7.4*sin(-52.5)},{cos(52.5)*7.4}) {{\scriptsize $x$}};

\end{tikzpicture}

\end{vwcol}
\vskip+0.5cm

\noindent\textbf{8. } If there remains a $(2m+3)$-set $H^{2m+3}(x)$ with non-zero charge, then in the pair $(H^{m+1}_l(x),$ $ H^{m+2}_s(x-1))$  one set is missing. We transfer the charge of $H^{2m+3}(x)$ to the missing set.\\

Let us check that no $(m+1)$- or $(m+2)$-set got too much charge. An $(m+2)$-set $H^{m+2}_s(x)$ could have gotten ${n\choose m-3}< \frac 1{12}{n\choose m+2}=\frac 12 w(H^{m+2}_s(x))$ charge at step 8, and did not get any charge before. The charge of any other $(m+2)$-set also does not exceed its charge. As for left and right $(m+1)$-sets, each of them could get at most ${n\choose m-3}$ at step 8, which together with the charge accumulated on step 5, gives at most $\frac 12{n\choose m-1}+{n\choose m-3}< \Big(\frac 12+\frac{m-1}{2m+2}\Big){n\choose m-1} = \frac m{m+1}{n\choose m-1},$ which is equal to the weight of each of these sets. Therefore, the left and right $(m+1)$-sets are not overcharged.

If an $(m+1)$-set $H^{m+1}(x)$ got the charge from both $(2m+2)$-sets, then none of the $m$-sets $H^{m}(x+m), H^{m}(x-m-1)$ that together with $H^{m+1}(x)$ form an interval of length $2m+1$ are in the family. In this case $H^{m+1}(x)$ does not appear in the later stages, and has charge $2{n\choose m-2}<{n\choose m-1}=\frac 12w(H^{m+1}(x))$. \\

\textsc{Stage C. Transferring charge from $m$-sets.}\\

\noindent\textbf{9. } Note that, at this point, if a set $F$ has non-zero charge then $m\le |F|\le 2m+1$ holds. Thus, we have to take care of the $m$-sets. Recall that all $m$-sets in $\mathcal H$ are intervals (arcs) on the circle.

If there are two adjacent $m$-sets $H^m(x), H^m(x-m)$ that are both in $\ff$, then we transfer the charge of one of them to the missing $H^{2m}(x)$. Note that in this case $H^{2m}(x)$ has zero charge at the beginning of Stage C. Indeed, it could have been charged only on steps 2 and 5, and in both cases one of the sets $H^m(x), H^m(x-m)$ should have been forbidden (see the corresponding figures). The charge transferred is ${n\choose m} = {n\choose 2m} = w(H^{2m}(x))$. We are not going to transfer any more charge to the $2m$-sets. From now on we assume that in each triple of disjoint interval $m$-sets at most one has non-zero charge. Note that this implies that there remain at most $m$ arcs of length $m$ that have positive charge.
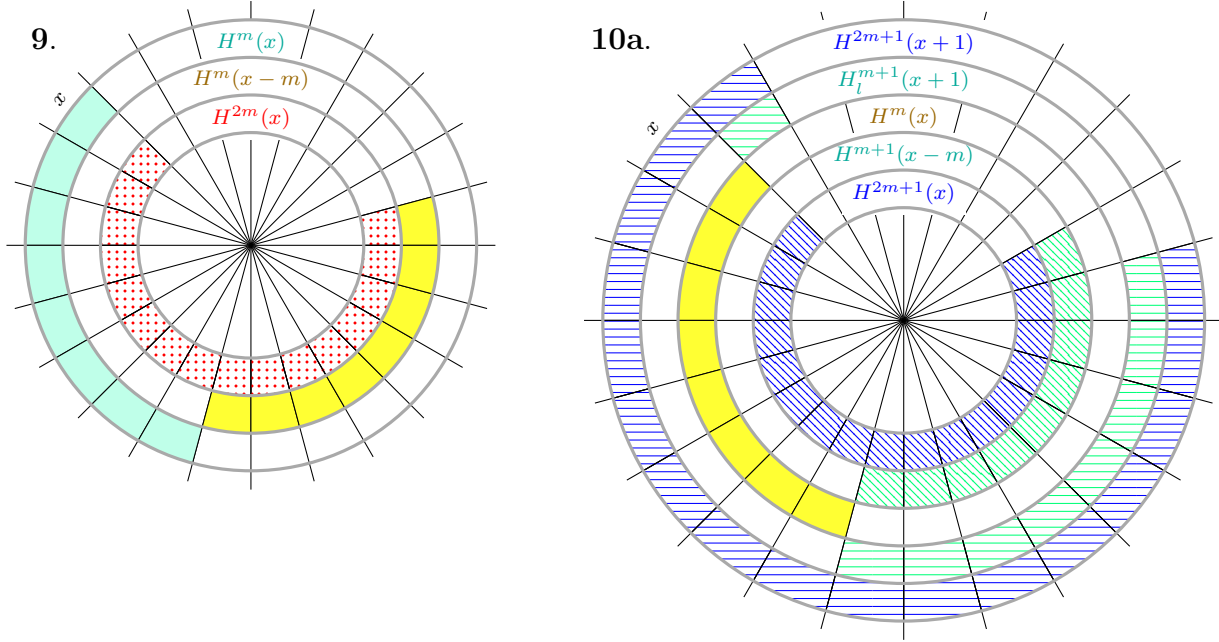
\begin{vwcol}[widths={0.4,0.6}, justify=flush, rule=0pt]
\begin{tikzpicture}[scale=0.5]

\foreach \x in {1,...,24}
\draw (0,0) -- (15*\x:6.5);


\foreach \x in {9,10,11,12,13,14,15,16}
\filldraw[fill=Aquamarine!50, draw=black, opacity=1] (15*\x:5) -- (15*\x:6) arc (15*\x:15*(\x+1):6) -- (15*\x+15:5) arc (15*\x+15:15*\x:5);
\node[fill=white] at (0,5.4) {{\scriptsize $\textcolor{JungleGreen}{H^{m}(x)}$}};



\foreach \x in {0,17,18,19,20,21,22,23}
\filldraw[fill=Yellow!80, draw=black, opacity=1] (15*\x:4) -- (15*\x:5) arc (15*\x:15*(\x+1):5) -- (15*\x+15:4) arc (15*\x+15:15*\x:4);
\node[fill=white] at (0,4.4) {{\scriptsize $\textcolor[rgb]{0.6,0.4,0}{H^{m}(x-m)}$}};




\foreach \x in {0,9,10,11,12,13,14,15,16,17, 18,19,20,21,22,23}
\filldraw[pattern=dots, pattern color=Red, draw=black] (15*\x:3) -- (15*\x:4) arc (15*\x:15*(\x+1):4) -- (15*\x+15:3) arc (15*\x+15:15*\x:3);

\node[fill=white] at (0,3.4) {{\scriptsize $\textcolor{red}{H^{2m}(x)}$}};


\foreach \x in {3,4,5,6}
\draw[very thick,DarkGrey] (0,0) circle (\x);

\node at (-5.5,5.5) {$\mathbf{9.}$};
\node[rotate=52.5] at ({6.4*sin(-52.5)},{cos(52.5)*6.4}) {{\scriptsize $x$}};

\end{tikzpicture}

\begin{tikzpicture}[scale=0.5]

\foreach \x in {1,...,24}
\draw (0,0) -- (15*\x:8.5);


\foreach \x in {9,10,11,12,13,14,15,16}
\filldraw[fill=Yellow!80, draw=black, opacity=1] (15*\x:5) -- (15*\x:6) arc (15*\x:15*(\x+1):6) -- (15*\x+15:5) arc (15*\x+15:15*\x:5);
\node[fill=white] at (0,5.4) {{\scriptsize $\textcolor[rgb]{0.6,0.4,0}{H^{m}(x)}$}};



\foreach \x in {0,1,17,18,19,20,21,22,23}
\filldraw[pattern=north west lines, pattern color=JungleGreen, draw=black, opacity=1] (15*\x:4) -- (15*\x:5) arc (15*\x:15*(\x+1):5) -- (15*\x+15:4) arc (15*\x+15:15*\x:4);
\node[fill=white] at (0,4.4) {{\scriptsize $\textcolor{JungleGreen}{H^{m+1}(x-m)}$}};




\foreach \x in {9,10,11,12,13,14,15,16,17,18,19,20,21,22,23,0,1}
\filldraw[pattern=north west lines, pattern color=blue, draw=black] (15*\x:3) -- (15*\x:4) arc (15*\x:15*(\x+1):4) -- (15*\x+15:3) arc (15*\x+15:15*\x:3);

\node[fill=white] at (0,3.4) {{\scriptsize $\textcolor{blue}{H^{2m+1}(x)}$}};



\foreach \x in {0,8,17,18,19,20,21,22,23}
\filldraw[pattern=horizontal lines, pattern color=JungleGreen, draw=black, opacity=1] (15*\x:6) -- (15*\x:7) arc (15*\x:15*(\x+1):7) -- (15*\x+15:6) arc (15*\x+15:15*\x:6);
\node[fill=white] at (0,6.4) {{\scriptsize $\textcolor{JungleGreen}{H_l^{m+1}(x+1)}$}};




\foreach \x in {8,9,10,11,12,13,14,15,16,17,18,19,20,21,22,23,0}
\filldraw[pattern=horizontal lines, pattern color=blue, draw=black, opacity=1] (15*\x:7) -- (15*\x:8) arc (15*\x:15*(\x+1):8) -- (15*\x+15:7) arc (15*\x+15:15*\x:7);
\node[fill=white] at (0,7.4) {{\scriptsize $\textcolor{blue}{H^{2m+1}(x+1)}$}};



\foreach \x in {3,4,5,6,7,8}
\draw[very thick,DarkGrey] (0,0) circle (\x);

\node at (-7.5,7.5) {$\mathbf{10a.}$};

\node[rotate=52.5] at ({8.4*sin(-52.5)},{cos(52.5)*8.4}) {{\scriptsize $x$}};

\end{tikzpicture}
\end{vwcol}\vskip+1cm

\noindent$\mathbf{10.\ }$  The charge of the remaining $m$-sets we are going to distribute among the $(m+1)$-sets and $(2m+1)$-sets. Define $A:=\{x\in [n]: H^m(x)\text{ has non-zero charge}\}$ and define $a:=|A|$. We aim to show that the total remaining capacity of $(m+1)$- and $(2m+1)$-sets is at least $a{n\choose m}$, which is the total charge of all non-discharged $m$-sets. This will conclude the proof of \eqref{eq16}. Having a non-discharged $H^{m}(x)$ implies that in each of the following four pairs one set is missing (see the figures 10a and 10b, where all four pairs are represented):
\begin{vwcol}[widths={0.4,0.6}, justify=flush, rule=0pt]
\begin{minipage}{65mm}
\begin{itemize}\item[]
\item $H^{m+1}(x-m),H^{2m+1}(x)$,
\item $H^{m+1}_l(x+1), H^{2m+1}(x+1)$,
\item $H^{m+1}_r(x+m), H^{2m+1}(x+m)$.
\item $H^{m+1}(x+m+1)$, $H^{2m+1}(x+m+1).$
\end{itemize}
\end{minipage}
\begin{tikzpicture}[scale=0.5]

\foreach \x in {1,...,24}
\draw (0,0) -- (15*\x:8.5);


\foreach \x in {9,10,11,12,13,14,15,16}
\filldraw[fill=Yellow!80, draw=black, opacity=1] (15*\x:5) -- (15*\x:6) arc (15*\x:15*(\x+1):6) -- (15*\x+15:5) arc (15*\x+15:15*\x:5);
\node[fill=white] at (0,5.4) {{\scriptsize $\textcolor[rgb]{0.6,0.4,0}{H^{m}(x)}$}};



\foreach \x in {0,...,8}
\filldraw[pattern=north west lines, pattern color=JungleGreen, draw=black, opacity=1] (15*\x:4) -- (15*\x:5) arc (15*\x:15*(\x+1):5) -- (15*\x+15:4) arc (15*\x+15:15*\x:4);
\node[fill=white,rotate=30] at ({4.4*sin(30)},{-4.4*cos(30)}) {{\scriptsize $\textcolor{JungleGreen}{H^{m+1}(x+m+1)}$}};




\foreach \x in {0,...,16}
\filldraw[pattern=north west lines, pattern color=blue, draw=black] (15*\x:3) -- (15*\x:4) arc (15*\x:15*(\x+1):4) -- (15*\x+15:3) arc (15*\x+15:15*\x:3);

\node[fill=white,rotate=30] at ({3.4*sin(30)},{-3.4*cos(30)}) {{\scriptsize $\textcolor{blue}{H^{2m+1}(x+m+1)}$}};



\foreach \x in {1,2,3,4,5,6,7,8,17}
\filldraw[pattern=horizontal lines, pattern color=JungleGreen, draw=black, opacity=1] (15*\x:6) -- (15*\x:7) arc (15*\x:15*(\x+1):7) -- (15*\x+15:6) arc (15*\x+15:15*\x:6);
\node[fill=white,rotate=20] at ({6.4*sin(20)},{-6.4*cos(20)}) {{\scriptsize $\textcolor{JungleGreen}{H_r^{m+1}(x+m)}$}};




\foreach \x in {1,...,17}
\filldraw[pattern=horizontal lines, pattern color=blue, draw=black, opacity=1] (15*\x:7) -- (15*\x:8) arc (15*\x:15*(\x+1):8) -- (15*\x+15:7) arc (15*\x+15:15*\x:7);
\node[fill=white,rotate=20] at ({7.4*sin(20)},{-7.4*cos(20)}) {{\scriptsize $\textcolor{blue}{H^{2m+1}(x+m)}$}};



\foreach \x in {3,4,5,6,7,8}
\draw[very thick,DarkGrey] (0,0) circle (\x);

\node at (-7.5,7.5) {$\mathbf{10b.}$};

\node[rotate=52.5] at ({8.4*sin(-52.5)},{cos(52.5)*8.4}) {{\scriptsize $x$}};

\end{tikzpicture}
\end{vwcol}
\vskip+2cm
Let us denote by $\hh^{q}(A)$ the set of all $q$-element sets that appear in the list above for some $x\in A$. We call all such pairs of subsets as listed above the \underline{forbidden pairs}.

We note that no left or right missing $(m+1)$-sets with  non-zero charge could appear in the list above. Indeed, a left or right $(m+1)$-set $H$ could have gotten some charge at steps 5 and 8, and in both cases the only interval $m$-set disjoint with $H$ must be missing from the family (see the corresponding figures).

An interval $(m+1)$-set may appear in at most $2$ forbidden pairs. If it was charged on the previous steps, then it can appear in at most $1$ pair. Indeed, it could have gotten some charge at step 7 only, and then one of the adjacent interval $m$-sets is not in the family. Moreover, as we have mentioned at the end of Stage B, if it got charged twice, then it cannot appear in the list above.

Let us use the following notation: $A+i:=\{\alpha+i: \alpha\in A\}$. To further analyze the situation, we consider an auxiliary bipartite graph $G=(U\cup V, E)$. Here $U=A\cup(A+1)\cup(A+m)\cup (A+m+1)$ corresponds to the $(2m+1)$-sets: \begin{equation*} U:=\{x: H^{2m+1}(x)\in \hh^{2m+1}(A)\}.\end{equation*} The set $V$ consists of three parts:
\begin{align*}
V_1 :=& \{x: H^{m+1}_l(x)\in \hh^{m+1}(A)\},\ \ \ \ \  \ V_1 = A+1;\\
V_2 :=& \{x: H^{m+1}_r(x)\in \hh^{m+1}(A)\}, \ \ \ \ \  \ V_2 = A+m;\\
V_3 :=& \{x: H^{m+1}(x)\in\hh^{m+1}(A)\},\ \ \ \ \ V_3 = (A-m)\cup (A+m+1);\\
V:=& (V_1,1)\sqcup (V_2,2)\sqcup (V_3,3).
\end{align*}
The set of edges $E$ consists of all pairs of vertices from $U$ and $V$, that correspond to a forbidden pair of a $(2m+1)$-set and an $(m+1)$-set. We also assign weights to vertices, equal to the capacity of the corresponding sets (the amount of charge they can still receive without surpassing their weight).

By the definition of a forbidden pair, there is a family $\mathcal S\subset\hh^{2m+1}(A)\cup \hh^{m+1}(A)$, which contains at least one subset from each forbidden pair and is disjoint from $\ff$. We want to lower bound the capacity of any such family. If this lower bound is at least $a{n\choose m}$, then we are done: we can redistribute the charge of the $m$-sets between the sets of the family $S$. In terms of the bipartite graph $G$, this is a problem of lower bounding the size of a minimal weight vertex cover. Thus, the following lemma concludes the proof of \eqref{eq16} and thereby of the bound (\ref{eq7}) for $n=3m$.

\begin{lem}\label{lembp}The minimal weight of a vertex cover in $G$ is at least $a{n\choose m}$.
\end{lem}
\begin{proof} Let us start with the analysis of the structure of the graph. First, $|E| = 4a$. Indeed, each vertex $x\in A$ gives rise to four forbidden pairs, and, therefore, four edges of $G$. Moreover, clearly, all the pairs are different.

Next, the degree of any vertex in $(V_1,1)$ or $(V_2,2)$ is $1$. Indeed, a left (as well as right) $(m+1)$-set is disjoint with exactly one interval $m$-set, which together defines uniquely the forbidden pair. The degree of each vertex in $(V_3,3)$ is either $1$ or $2$: for each interval $(m+1)$-set there are exactly two interval $m$-sets, together with which it forms an interval of length $(2m+1)$. Moreover, recalling the discussion after Stage 2, all left and right $(m+1)$-sets have zero charge, and if an interval $(m+1)$-set has non-zero charge, then the degree of a corresponding vertex in $G$ is $1$.

The degree of any vertex in $U$ is also either $1$ or $2$. Indeed, for any $(2m+1)$-set $H^{2m+1}(x)$ there are four interval $m$-sets that it contains that would possibly give rise to a forbidden pair: $H^{m}(x), H^m(x-1), H^{m}(x-m), H^m(x-m-1)$. These four sets split into two pairs of adjacent $m$-sets, thus, at this stage we can have at most one out of each pair (cf. step 9). Another important fact about $U$ is that $|U|\ge 2a+1$. Indeed, $|A\cup (A+m)|=2a$ due to the fact that there are no $x_1,x_2\in A$ with $x_1-x_2=m$. Then $(A+1)\cup (A+m+1) = (A\cup (A+m))+1$, and so $(A+1)\cup (A+m+1) \neq A\cup (A+m)$. (To see this, consider the {\it clockwise boundary} of $A\cup (A+m)$: the elements $x\in [n]$ which satisfy $x\in A\cup (A+m)$, but $x+1\notin A\cup (A+m)$. Clearly, the clockwise boundary of $A\cup (A+m)$ is non-empty, and is contained in $(A+1)\cup (A+m+1)$.) Therefore, $|U|>|A\cup (A+m)|=2a$.

We finish the description of the graph by recalling the weights of the vertices. All vertices in $U$ have weights $w_0:={n\choose m-1}$. All vertices in $(V_1,1)$ and $(V_2,2)$ have weight $w_1:=\frac m{m+1}{n\choose m-1}$. Note the inequality $w_1> \frac 56 w_0$, valid for $m\ge 6$. The vertices in $(V_3,3)$ that have degree $2$ correspond to interval $(m+1)$-sets with zero charge, and so have weight $2{n\choose m-1}$. The vertices in $(V_3,3)$ of degree $1$ have weight at least $w_2:=2{n\choose m-1}-{n\choose m-2}\ge \frac{3}2 {n\choose m-1}$.  \\

Let $W$ be a vertex cover in $G$. Let $k:=W\cap ((V_1,1)\cup (V_2,2))$. Assume first that $k\ge 3$. Then, removing these $k$ vertices from $G$, we have still $4a-k$ edges to cover. In the remaining graph we would have to spend at least $\frac {w_0}2$ of weight per edge (see the possible  weights and degrees of the vertices), and the total weight of $W$ would be $$kw_1+\frac {(4a-k)w_0}2 = 2aw_0+k(w_1-\frac {w_0}2)> (2a+1)w_0.$$
We finish the proof in this case by the following inequality, valid for any $a\le m$: \begin{equation}\label{eq15}(2a+1)w_0 = (2a+1){n\choose m-1} = \frac {(2a+1)m}{2m+1}{n\choose m}\ge a{n\choose m}.\end{equation}

Assume that $k\le 2$. Note that for any subset $R\subset (V_3,3)$ the weight of $R$ is at least as big as the weight of $N(R)$, the neighborhood of $R$. Therefore, there exist a vertex cover $W$ of minimal weight, which does not use any vertices from $(V_3,3)$ (and which contains $N((V_3,3))$).
If $|U|\ge 2a+2$, then we are done as well: In the worst case, we take in the vertex cover $k$ vertices from $(V_1,1)\cup (V_2,2)$ and $2a+2-k\ge 2a$ vertices from $U$. The total weight of the vertex cover in this situation is $2aw_0+2w_1>(2a+1)w_0\ge a{n\choose m}$. If $k=0$, then we are good again: we have $W=U$  and it has weight at least $(2a+1)w_0$.

We are left with the following case: $k\in \{1,2\}$ and $|U|=2a+1$. Since $|U|=2a+1$, the set $A\cup (A+m)$ forms an interval of length $2m$. Otherwise, it would have had at least two points of clockwise border, and the size of $U$ would have been at least $2a+2$. But then $A$ itself must form an interval of length $m$, w.l.o.g., $[1,m]$. Then $U=[1,2m+1]$, and the only two sets that correspond to vertices in $U$ of degree 1 are $H^{2m+1}(1)$ and $H^{2m+1}(2m+1)$. But both are connected to a vertex in $(V_3,3)$ (corresponding to the set $H^{m+1}(2m+1)$).

We get that the vertices in $W\cap ((V_1,1)\cup (V_2,2))$ are connected to vertices in $U$ of degree $2$. Therefore, we would have to take at least $2a+2-k$ vertices from $U$ in the vertex cover, and so the weight of the vertex cover would be at least $(2a+1)w_0+(kw_1-(k-1)w_0)>(2a+1)w_0$. This concludes the proof of the lemma.
\end{proof}
\vskip+0.2cm
\textbf{Extremal families}
\vskip+0.1cm
Claim~\ref{clach2} implies that in any extremal family $\ff$ all sets of sizes $k\in [m+2,2m-1]$ are present. This immediately implies that no sets of size $k\le m-3$ or $k\ge 2m+4$ are in the family.

Further, we can see that if $\ff$ did not contain a set of size $(m+1)$, then for some permutation it would have been a missing right set in $\hh$. But right sets until step 10 were charged only at step 5, and still had some capacity left. At the same time, having a right $(m+1)$-set at step 10 in the vertex cover in Lemma~\ref{lembp} implies that the vertex cover has charge strictly greater than $a{n\choose m}$. In both cases we conclude that $\ff$ was not of maximal possible size, contradicting the initial assumption. Therefore, all $(m+1)$-sets are contained in $\ff$, and, as a corollary, no $(2m+2)$- and $(2m+3)$-sets are contained in $\ff$, as well as no $(m-2)$-sets.

If an $(m-1)$-set is contained in $\ff$, then it appears in $\hh$ for some permutation, and it implies, together with the fact that all $(m+1)$-sets are in $\ff$, that some $2m$-set $H$ of non-zero weight is missing from $\hh$. It means that $H$ got some charge until step 10 and did not participate in step 10. But any such set was not fully charged, again contradicting the maximality assumption.

Therefore, $\cup_{k=m+1}^{2m-1}{[n]\choose k}\subset \ff\subset \cup_{k=m}^{2m+1}{[n]\choose k}$. Now we have to look more carefully on steps 9 and 10, in particular on Lemma~\ref{lembp}. There are two cases in which a vertex cover can have total weight exactly $a{n\choose m}$.

The first case is simple: $a=0$. Then it is clear that there were no $m$-sets in $\hh\cap\ff$: we discharged $m$-sets only on step 9, but we discharged only one out of each pair of adjacent sets.

In the second case $a=m$, and the set $A$ from Lemma~\ref{lembp} must be an interval $[x-m+1,x]$ for some $x$. Indeed, if $a<m$, then the inequality in \eqref{eq15} becomes a strict inequality. Moreover, $|U|=2a+1$ only if $A$ is an interval. The question we have to decide in this case is whether it was possible that some $m$-sets were discharged at step 9 for an {\it extremal} $\ff$.
Actually, if there was a (last remaining) pair of adjacent $m$-sets $H_1,H_2$ of non-zero charge at step 9, we could choose to transfer the charge of any of them to $H_1\cup H_2$. But one of the choices would lead to the set $A$ which is not an interval of length $m$, with the only exception: the set $A':=\{y\in [n]: H^m(y)\in \ff\}$ satisfies $A'=[x,x+m]$ for some $x\in[n]$ (note that $|A'|=m+1$). Putting this exception aside for a moment and assuming that we made the choice that leads to a non-interval $A$, we get a contradiction. Indeed, the family of non-discharged $m$-sets at step 10 does not fall into any of the two cases above. This also implies that all $2m$-sets are in $\ff$.

Suppose that $A'=[x,x+m]$. Since all $(m+1)$-sets are in $\ff$, we know that all sets $H^{2m+1}(y), y=x,\ldots, x+2m+1$ are missing from $\ff$. Their total capacity is $(2m+2){n\choose m-1}>m{n\choose m}$, which is the charge of all sets in $A'$ but one (which charge is transferred to $H^{2m}(x+2m-1)$). Therefore, it is impossible to get equality in this case.

The argument above shows that $\ff\supset \cup_{k=m+1}^{2m}{[n]\choose k}$ and that for any $\hh$ either $A'=\emptyset$ or $A'=[x,x+m-1]$ for some $x$. This means that the family $\mathcal G:=\hh\cap\ff\cap \big[{[n]\choose m}\cup {[n]\choose 2m+1}\big]$ can be of only two forms:
$$\mathcal G=\hh\cap{[n]\choose 2m+1}\ \ \ \text{ or }\ \ \ \mathcal G = \hh\cap\Big[\big\{H\in {[n]\choose m}:x\in H\big\}\cup \big\{H\in {[n]\choose 2m+1}:x\notin H\big\}\Big]$$
for some $x\in [n]$. Here the first form corresponds to the case $a=0$, and the second form corresponds to the case $a=m$.

To conclude the proof of Theorem~\ref{thm1} for $n=3m$, we need to show that the first form corresponds to the family $\mathcal K(3m)$, while the second form corresponds to the family $\tilde{\mathcal K}_x(3n)$ for some $x\in[n]$, and that no other family is extremal.

Suppose the family $\mathcal G$ has the second form for at least one $\hh$ for some $x\in[n]$. We claim that then $\ff\cap{[n]\choose m}= \{F\in {[n]\choose m}:x\in F\}$. Let us first prove that \begin{equation}\label{eq18}\ff\cap{[n]\choose m}\subset \{F\in {[n]\choose m}:x\in F\}.\end{equation}
Take two $m$-sets $F_1,F_2\in \ff$, such that $F_1\cap F_2=\{x\}$. (The existence of such two sets follows from the assumption 3 lines above.) Then, for each permutation that makes both of them intervals, all other $m$-intervals in these permutations that contain $x$ belong to $\ff$. In other words, $$\mathcal Q:=\Big\{Q\in {[n]\choose m}:x\in Q, Q\subset F_1\cup F_2\Big\}\subset \ff.$$

Assume that there is an $m$-set $G$ in $\ff$, such that $x\notin G$. Take a permutation $\sigma$ such that in it $G$ becomes an interval and both elements adjacent to $G$ are not from $F_1\cup F_2$. For $\sigma$ the family $\mathcal G$ must be of the second form. This and the choice of the elements adjacent to $G$ guarantees that there is a set $G'\in \mathcal G$ such that $G'\not\subset F_1\cup F_2$. But there exists $Q\in \mathcal Q$ such that $G'\cap Q=\emptyset$. Then, taking a permutation that makes both of them intervals, we arrive at a contradiction with the possible forms the family $\mathcal G$ may have for that permutation.

We conclude that \eqref{eq18} holds. To prove the inclusion in the other direction, assume that there is an $m$-set $H\notin \ff$, such that $x\in H$. Then take any permutation that makes both $H$ and $F_1$ intervals. We know that the corresponding $\mathcal G$ must be of the second form, with the center in $x$. This is a contradiction.


We conclude that either $\ff\cap {[n]\choose m}=\emptyset$, and in that case $\ff=\mathcal K(3m)$, or $\ff\cap {[n]\choose m}=\{F\in {[n]\choose m}:x\in F\}$ for some $x\in [n]$, and in this case $\ff=\tilde{\mathcal K}_x(3m)$. The proof of Theorem~\ref{thm1} for $n=3m$ is complete.

\section{The proof of Theorem~\ref{thm3}}
Assume that $m\ge 4$. It is possible to prove Theorem~\ref{thm3} using charging-discharging method. For a change, in this section we give a proof with a somewhat different and hopefully simpler analysis.

We are also going to average over the choice of a particular $\hh$. It contains three groups, based on an equipartition $[3m]=H_1^m\sqcup H_2^m\sqcup H_3^m$. For simplicity we assume that $i\in H_i^m$ for each $i\in[3]$. In what follows we define the $i$-th group $\hh_i$. For the definition of the sets choose $j,k$ such that $\{i,j,k\}:=[3]$. The group contains
\begin{itemize}
\item one $(m-1)$-set: $H_i^{m-1}:= H_i^m\setminus \{i\}$;
\item one $m$-set: $H_i^{m}$;
\item two $(m+1)$-sets: $H_i^{m+1}(j):=H_i^{m+1}\cup \{j\}$, each of weight ${n\choose m-1}+1$;
\item one $(2m-2)$-set $H_i^{2m-2}:=H_j^{m-1}\cup H^{m-1}_k$ of weight ${n\choose m-1}+1$;
\item one $2m$-set $H^{2m}_i:=H_j^m\cup H_k^m$;
\item one $(2m+1)$-set $H_i^{2m+1}:=[n]\setminus H_i^{m-1}$.
\end{itemize}
All non-specified weights are ${n\choose j}$ for sets of size $j$.

\begin{prop}\label{prop1} If $\ff_1,\ff_2,\ff_3$ are cross partition-free, then \begin{equation}\label{eq17}\sum_{i=1}^3\sum_{F\in \hh_i\setminus \ff_i}w(F)\ge 2{n\choose m}+5{n\choose m-1}.\end{equation}
\end{prop}
\begin{proof} We are going to distinguish several cases. As ${3m\choose m-1}/{3m\choose m} = \frac{m}{2m+1}<\frac 12$, arguing indirectly  we may assume that at least two of the corresponding $m$- and $2m$-sets are present in the $\ff_i$.\\

\textbf{a) Exactly two are present (and four are missing).} It is enough to find one more missing set. If $H_i^m\in \ff_i$, then not {\it both} of $H_j^{m+1}(k)\in \ff_j$ and $H_k^{2m+1}\in \ff_k$ can hold.

If $H^{2m}_k\in\ff_k$ then not both $H^{m+1}_i(j)\in\ff_i$ and $H_j^{m-1}\in\ff_j$ can hold.\\

\textbf{b) Exactly three are present (and three are missing).} \\

We need to find three more missing sets. Two $2m$-sets force three missing sets. Say, $H^{2m}_k$ and $H^{2m}_j$ are present. Then consider three pairs
\begin{align*}H_i^{m+1}(j)&\ \ - \ \ H_j^{m-1},\\
H_i^{m+1}(k)&\ \ - \ \ H_k^{m-1},\\
H_j^{m+1}(i)&\ \ - \ \ H_i^{m-1}.\end{align*}
In each of these pairs there is at least one missing set.

Similarly, two $m$-sets force three missing sets. Assuming that $H_i^m, H_j^m$ are present, we get that in each of the pairs below one set is missing.
\begin{align*}H_j^{m+1}(k)&\ \ - \ \ H_k^{2m+1},\\
H_k^{m+1}(j)&\ \ - \ \ H_j^{2m+1},\\
H_k^{m+1}(i)&\ \ - \ \ H_i^{2m+1}.\end{align*}

\textbf{c) Exactly two of the $m$- and $2m$-sets are missing.} \\

In this case we have to find 5 more missing sets. We have two subcases.\\

\textbf{c1) $H_j^m$, $H_i^m$ are missing.} \vskip+0.1cm

Then $H_i^{2m}, H_j^{2m}, H_k^{2m}$ are present, and in each of the pairs below one set is missing:
\begin{align*}H_j^{m+1}(i)&\ \ - \ \ H_i^{2m+1},\\
H_i^{m+1}(j)&\ \ - \ \ H_j^{2m+1},\\
H_k^{m+1}(i)&\ \ - \ \ H_i^{m-1},\\
H_j^{m+1}(k)&\ \ - \ \ H_k^{m-1},\\
H_k^{m+1}(j)&\ \ - \ \ H_j^{m-1}.\end{align*}

\textbf{c2) $H_i^m$, $H_i^{2m}$ are missing.} \vskip+0.1cm

Then $H_j^{m}, H_k^{m}, H_j^{2m}, H_k^{2m}$ are present, and in each of the pairs below one set is missing:
\begin{align*}H_j^{m+1}(i)&\ \ - \ \ H_i^{2m+1},\\
H_k^{m+1}(i)&\ \ - \ \ H_i^{m-1},\\
H_i^{m+1}(k)&\ \ - \ \ H_k^{m-1},\\
H_i^{m+1}(j)&\ \ - \ \ H_j^{m-1}.\end{align*}
These pairs provide at least four more missing sets. Assume that $H_j^{2m+1}\in\ff_j, H_k^{2m+1}\in\ff_k$ (otherwise, we are done). This, together with $H_j^{m}\in \ff_j, H_k^{m}\in \ff_k$ implies that $H_i^{m+1}(k)$ and $H_i^{m+1}(j)$ are missing. Thus, either one of $H^{m-1}_j, H^{m-1}_k$ is missing, and we are done, or both are present, and in this case $H^{2m-2}_i$ is missing. The proof is complete.
\end{proof}

Note that the weight of $(m+1)$- and $(2m-2)$-sets is greater than the weight of $(m-1)$-sets. Therefore, in case of equality in \eqref{eq17} the missing sets must be $5$ out of the altogether $6$ sets of sizes $m-1$ and $2m-1$, and $2$ of the altogether $6$ sets of sizes $m$ and $2m$. Thus all the $(m+1)$- and $(2m-2)$-sets must be present. \\

Next we may combine \eqref{eq17} with an obvious analogue of Claim~\ref{clach2} for cross partition-free families. We omit the calculations, that almost repeat the calculations of Claim~\ref{clach2}. We only remark that, in addition to the triples of sizes mentioned in the claim, we have to consider the triples $(s_1,s_2,s_3)=(m-2,m+1,2m-1), (m+1,m+1,2m+2), (m+1,m+2,2m+3)$.

As in Claim~\ref{clach2}, equality in \eqref{eq9} implies that all sets of sizes $[m+1,\ldots,2m-1]$ must be in each $\ff_i$. This implies that $\ff_i\subset \cup_{t=m-1}^{2m+1}{[n]\choose t}$. Moreover, from the proof of Proposition~\ref{prop1} it follows that equality in \eqref{eq9} is possible only if $\ff_i$ does not contain any sets of size $m-1$ (otherwise, some $(m+1)$-sets are missing), and if exactly two out of $m$- and $2m$-sets are missing from $\cup_{i=1}^3\mathcal H_i\cap \ff_i$ for each choice of a triple of disjoint $m$-sets (so that we fall either in Case c1 or c2).

If we fall into Case c2, but all $(m+1)$-sets are present, then all $(m-1)$-sets are missing, but also all $(2m+1)$-sets are missing: $H_i^{2m+1}=H_j^m\cup H_k^{m+1}(i), H_j^{2m+1}=H_k^m\cup H_i^{m+1}(j), H_k^{2m+1}=H_j^m\cup H_i^{m+1}(k)$. Therefore, the equality cannot hold in this case.\\

We conclude that we are in the situation c1 and in each triple there in exactly one present $m$-set and $(2m+1)$-set, moreover, both belong to the same family. We are only left to prove that it must come from the same family for each triple. Note that $\cup_{t=m+1}^{2m}{[n]\choose t}\subset \ff_i$ for each $i\in[3]$.

Assume that at least two out of $\ff_i^{(2m+1)}$, $i=1,2,3$, are nonempty. Then for each $i\ne j$ and $F_i\in\ff_i^{(2m+1)}$, $F_j\in \ff_j^{(2m+1)}$ we have $F_i\cup F_j\ne [n]$. Otherwise, there exists a triple in which $F_i$ is the set $H^{2m+1}_i$ and $F_j=H^{2m+1}_j$, a contradiction with the form of $\cup_k \ff_k\cap \hh_k$.

Define $\mathcal Q_i:=\{[n]\setminus F: F\in \ff^{(2m+1)}_i\}$. Then the previous paragraph implies that for any $i\ne j$ $\mathcal Q_i$ and $\mathcal Q_j$ are \underline{cross-intersecting}: for any $Q_i\in \mathcal Q_i, Q_j\in \mathcal Q_j$ we have $Q_i\cap Q_j\ne \emptyset$. In \cite{FK1} the following useful inequality was proved (see Theorem 9 and Corollary 12 in \cite{FK1}): If $\mathcal G_1, \mathcal G_2\subset {[n]\choose t}$ are cross-intersecting and $|\mathcal G_1|\ge |\mathcal G_2|$, then for any $c\ge 1$ and $n\ge 2t$ one has
\begin{equation}\label{eq20}|\mathcal G_1|+c|\mathcal G_2|\le \max\Big\{{n\choose t},(c+1){n-1\choose t-1}\Big\}.
\end{equation}
W.l.o.g., assume that $|\mathcal Q_1|\ge |\mathcal Q_2|\ge |\mathcal Q_3|$. For $\epsilon=\frac 1m$ one has $(3+\epsilon){n-1\choose m-2}=(3+\epsilon)\frac{m-1}{3m}{n\choose m-1}<{n\choose m-1}$. Then, applying \eqref{eq20}, we get
$$\sum_{i\in[3]}|\mathcal Q_i|\le |\mathcal Q_1|+(2+\epsilon)|\mathcal Q_2|\le \max \Big\{{n\choose m-1},(3+\epsilon){n-1\choose m-2}\Big\}={n\choose m-1}.$$
Moreover, the first inequality above is strict unless $\mathcal Q_2=\mathcal Q_3 =\emptyset$. Therefore, we conclude that the equality in \eqref{eq9} may hold only if $\mathcal Q_1={[n]\choose m-1}$, and therefore if $\ff_1,\ff_2,\ff_3$ have the form as given in Example 4.

\section{Discussion}
In this paper we have completely settled the problem of determining the maximum size of a partition-free family $\ff$, as well as the multi-family analogue of this question. One natural direction to extend these results is to study $r$-partition-free families, defined in the introduction, as well as to study their $r$-partite analogues. Another natural generalization of partition-free families, that was overlooked so far, are the $r$-box-free families (also defined in the introduction).

More generally, we may ask the following question. Given a poset $(P,<)$, what is the largest size of a family $\ff\subset 2^{[n]}$, which does not contain a \underline {disjoint representation} of $(P,<)$? We say that $\ff$ contains a disjoint representation of $(P,v)$ if $\ff$ contains a subfamily $\mathcal H$ and there is a bijective function $f:\mathcal H\to P$ such that for any $H_1,H_2 \in \mathcal H$ $f(H_1)<f(H_2)$ only if $H_1\subset H_2$, with the additional condition that any two sets from $\mathcal H$ corresponding to minimal elements of $(P,<)$ are disjoint. We may also require the disjoint representations to be \underline{exact}, that is, to require that for every non-minimal $S\in \mathcal H$ we have $S = \cup_{i:f(S_i)<f(S)} S_i$.  In this terms, the question we addressed in this paper asks for the largest $\ff\subset 2^{[n]}$ without an exact disjoint representation of a poset on the elements $\{a,b,c\}$ with relations $a>b, a>c$. \\

We say that a family $\ff$ is \underline{$t$-pseudo partition-free}, if $\ff$ does not contain three sets $A,B,C$ with $A\cup B=C$ and $|A\cap B|<t$. One natural example of a $t$-pseudo partition-free family is $\{F\subset [n]: m\le |F|\le 2m-t\}$.  The following sharp result may be proved using a direct generalization of Kleitman's argument \cite{Kl2}.

\begin{thm}Let $n = 3m-t+2$, $1\le t\le \frac m8$. Then any $t$-pseudo partition-free family $\ff$ satisfies \begin{equation*} |\mathcal F|\le \sum_{t=m+1}^{2m-t+2} {n\choose t}.\end{equation*} \end{thm}
Below we give an outline of the proof of this theorem.
\begin{proof}[Sketch of the proof]
A natural variant of \eqref{eq1} for $t$-pseudo partition free families would state that for any $s_1,s_2,s_3$, such that $s_1+s_2+s_3 = n+t-1$, one has the following inequality:
\begin{equation}\label{eq19}\sum_{i=1}^3 \frac {y^{s_i}+y^{n-s_i}}{{n\choose s_i}}\ge 2.\end{equation}
Indeed, just take three random sets $A,B,C$ of sizes $s_1,s_2,s_3$, respectively, with $S = A\cap B = B\cap C = A\cap C$, $|S|=t-1$. Then note that among $A,B,C, A\cup B, A\cup C, B\cup C$ there are at least two sets that are missing from $\ff$. Finally, average over the choice of $A,B,C$.\\

\begin{vwcol}[widths={0.5,0.37}, rule=0pt]

Next, we reason as in Section~3. We apply \eqref{eq19} for different triples of $s_i$, listed in Table~3. We sum up all the obtained inequalities and multiply them by the corresponding ${n\choose s_1}$ (except for the first one, which we multiply by $\frac 12 {n\choose m}$). Now we know that all the coefficients in front of $y^r$ for $r\le m$ and $r\ge 2m-t+3$ are equal to $1$. We only need to make sure that the coefficients in front of $y^r$ for $m+1\le r\le 2m-t+2$ are also at most $1$.

\begin{tabular}[t]{|c|c|c|}
\hline
$s_1$&$s_2$&$s_3$\\
\hline
$m$&$m$&$m+1$\\
$m-1$&$m+1$&$m+1$\\
$m-2$&$m+1$&$m+2$\\
$m-3$&$m+2$&$m+2$\\
$\vdots$&$\vdots$&$\vdots$\\
$m-j$&$m+\lfloor \frac {j+1}2\rfloor$&$m+\lceil \frac {j+1}2\rceil$\\
$\vdots$&$\vdots$&$\vdots$\\
$0$&$\lfloor \frac {n+t-1}2\rfloor$&$\lceil \frac {n+t-1}2\rceil$\\
\hline
\end{tabular}
\vskip-0.05cm \phantom{a}\qquad \qquad \quad Table 3
\end{vwcol}
We have $\frac{{n\choose m+1}}{{n\choose m}} = \frac {2m-t+2}{m+1}\ge\frac{2m-\frac m8+2}{m+1}>\frac{15}8$. Analogously, $\frac{{n\choose m-j+1}}{{n\choose m-j}}>\frac {15}8$ for any $j\ge 0$.
The coefficients in front of $y^{m+1}$ and $y^{2m-t+1}$ are equal to
$$\frac{\frac 12{n\choose m}+2{n\choose m-1}+{n\choose m-2}}{{n\choose m+1}}<\frac{\frac 12{n\choose m}+2{n\choose m-1}+{n\choose m-2}}{\frac{15}8{n\choose m}}=\frac 4{15}+\frac{\frac{16}{15}{n\choose m-1}+\frac 8{15}{n\choose m-2}}{{n\choose m}}\le $$
$$\frac 4{15}+\frac{\frac{16}{15}{n\choose m-1}+\frac 8{15}{n\choose m-2}}{\frac{15}8{n\choose m-1}}=\frac 4{15}+\frac{128}{225}+\frac{\frac {64}{225}{n\choose m-2}}{{n\choose m-1}}\le \frac 4{15}+\frac{128}{225}+\frac {256}{3375} = \frac{3076}{3375}<1. $$
The coefficients in front of $y^{m+1+j}$ and $y^{2m-t+1-j}$ for $j\ge 1$ are
$$\frac{{n\choose m-2j}+2{n\choose m-2j-1}+{n\choose m-2j-2}}{{n\choose m+j}}< \frac{4{n\choose m-2j}}{{n\choose m+j}}<4\Big(\frac8{15}\Big)^3<1.$$
\end{proof}

It would be interesting to find analogous results for $n=3m-t$ and $n=3m-t+1$, as well as to get a significant improvement of the bound $t\le \frac m8$.

\end{document}